\documentclass[a4paper,10pt, reqno]{amsart}

\usepackage{amsmath,amssymb,latexsym, amsfonts}
\usepackage{verbatim}
\usepackage{times}
\usepackage{mathbbol}
\usepackage[colorlinks=true, allcolors=black]{hyperref}
\usepackage{mathabx}
\usepackage{bbm}
\usepackage{pifont}
\usepackage{manfnt}

\usepackage{tikz}

\usepackage[all,ps]{xy}
\usepackage{color}

\usepackage{stmaryrd}

\usepackage{capt-of}

\usepackage[OT2, T1]{fontenc} 

\usepackage[main=english, russian]{babel}

\numberwithin{equation}{section}

\theoremstyle{plain}

\newtheorem{theorem}{Theorem}[section]

\newtheorem{prop}[theorem]{Proposition}
\newtheorem{lemma}[theorem]{Lemma}
\newtheorem{lem}[theorem]{Lemma}

\newtheorem{cor}[theorem]{Corollary}

\theoremstyle{definition}

\newtheorem{dfn}[theorem]{Definition}
\newtheorem{example}[theorem]{Example}

\newtheorem{rem}[theorem]{Remark}

\newcommand{\ahha}{{\scriptscriptstyle{A}}}

\newcommand{\uhhu}{{\scriptscriptstyle{U}}}

 \newcommand{\N}{{\mathbb{N}}}
 \newcommand{\C}{{\mathbb{C}}}
 \newcommand{\R}{{\mathbb{R}}}
 
 \newcommand{\Q}{{\mathbb{Q}}}

\newcommand{\ga}{\alpha} 

\newcommand{\gG}{\Gamma}
\newcommand{\gd}{\delta} 
\newcommand{\gD}{\Delta} 
\newcommand{\gve}{\varepsilon} 
  
\newcommand{\gvf}{\varphi}  

\newcommand{\gl}{\lambda} 
\newcommand{\gL}{\Lambda}
\newcommand{\go}{\omega}

\newcommand{\gvr}{\varrho} 

\newcommand{\gs}{\sigma} 
\newcommand{\gS}{\Sigma}
 
\newcommand{\gvt}{\vartheta} 

\newcommand{\cC}{{\mathcal C}}

\newcommand{\cL}{{\mathcal L}}
\newcommand{\cM}{{\mathcal M}}
\newcommand{\cN}{{\mathcal N}}
\newcommand{\cO}{{\mathcal O}}
\newcommand{\cP}{{\mathcal P}}

\newcommand{\cS}{{\mathcal S}}
\newcommand{\cT}{{\mathcal T}}

\newcommand{\cX}{{\mathcal X}}

\newcommand{\Alt}{\operatorname{Alt}}

\newcommand{\Hom}{\operatorname{Hom}}
\newcommand{\Der}{\operatorname{Der}}

\newcommand{\Tor}{{\rm Tor}}
\newcommand{\Ext}{{\rm Ext}}
\newcommand{\Cohom}{\operatorname{Cohom}}
\newcommand{\Cotor}{\operatorname{Cotor}}
\newcommand{\Coext}{\operatorname{Coext}}
\newcommand{\Cob}{\operatorname{Cob}}

\newcommand{\id}{{\rm id}}

\newcommand{\pr}{{\rm pr} \,}

\newcommand{\due}[3]{{}_{{#2 }} {#1}_{{ #3}}\,}    

\newcommand{\pl}{\partial}

\newcommand{{\Hl}}{{H^{\ell}}} 
\newcommand{{\mHop}}{{m_{H^{\rm op}}}} 
\newcommand{{\Hop}}{{H^{\rm op}}} 
\newcommand{{\mUop}}{{m_{U^{\rm op}}}} 
\newcommand{{\mUopp}}{{m_{\scriptscriptstyle{U^{\rm op}}}}} 
\newcommand{{\Uop}}{{U^{\rm op}}}
\newcommand{{\mVop}}{{m_{V^{\rm op}}}} 
\newcommand{{\Vop}}{{V^{\rm op}}}  
\newcommand{{\Ae}}{{A^{\rm e}}}
\newcommand{{\Be}}{{B^{\rm e}}}
\newcommand{{\Ue}}{{U^{\rm e}}}
\newcommand{{\He}}{{H^{\rm e}}}
\newcommand{{\Aop}}{{A^{\rm op}}}
\newcommand{{\Aope}}{({A^{\rm op}})^{\rm e}}
\newcommand{{\Aopl}}{{A^{\rm op}_\pl}}

\newcommand{{\Bop}}{{B^{\rm op}}}
\newcommand{{\Bopp}}{{\scriptscriptstyle{{B^{\rm op}}}}}
\newcommand{{\Bope}}{({B^{\rm op}})^{\rm e}}
\newcommand{{\Bpl}}{{B_\pl}}

\newcommand{{\op}}{{{\rm op}}}
\newcommand{{\coop}}{{{\rm coop}}}
\newcommand{{\sop}}{{*^{\rm op}}}
\newcommand{{\co}}{{{\rm co}}}

\newcommand{\kmod}{k\mbox{-}\mathbf{Mod}}                     %
\newcommand{\amoda}{A^{\rm e}\mbox{-}\mathbf{Mod}}                  %
\newcommand{\umod}{U\mbox{-}\mathbf{Mod}}                     
\newcommand{\yd}{{}^\uhhu_\uhhu\mathbf{YD}}                     

\newcommand{\comodu}{\mathbf{Comod}\mbox{-}U}         
\newcommand{\ucomod}{U\mbox{-}\mathbf{Comod}}         
\newcommand{\contramodu}{\mathbf{Contramod}\mbox{-}U}

\newcommand{\lact}{\smalltriangleright}                  
\newcommand{\ract}{\smalltriangleleft}
\newcommand{\blact}{\blacktriangleright}  
\newcommand{\bract}{\blacktriangleleft}

\newcommand{{\rra}}{\rightrightarrows}

\newcommand{{\lra}}{\ \longrightarrow \ }
\newcommand{{\lla}}{\ \longleftarrow \ }
\newcommand{{\lma}}{\ \longmapsto \ }

 \def\kasten#1{\mathop{\mkern0.5\thinmuskip
 \vbox{\hrule
       \hbox{\vrule
             \hskip#1
             \vrule height#1 width 0pt
             \vrule}%
       \hrule}%
 \mkern0.5\thinmuskip}}

\newcommand{\bx}{{\kasten{6pt}}}

\newcommand{{\bull}}{{\scriptscriptstyle{\bullet}}}
\newcommand{{\qqquad}}{{\quad\quad\quad}}
\newcommand{\Aopp}{{\scriptscriptstyle{\Aop}}}

\newsavebox{\foobox}

\newcommand{\pmact}{\mbox{ \raisebox{-1pt}{\ding{226}} }}
\newcommand{\mpact}{\mbox{ \raisebox{-1pt}{\ding{227}} }}
\newcommand{\sma}[1]{\raisebox{1pt}{${\scriptstyle ({#1})}$}}
\newcommand{\smap}{\raisebox{1pt}{${\scriptstyle (-)_{[+]}}$}}
\newcommand{\smam}{\raisebox{1pt}{${\scriptstyle (-)_{[-]}}$}}

\newcommand{\smadotp}{\raisebox{1pt}{${\scriptstyle (\cdot)_{[+]}}$}}
 \newcommand{\smadotm}{\raisebox{1pt}{${\scriptstyle (\cdot)_{[-]}}$}}

\newcommand{\mancino}{{\,\scalebox{0.7}{\rotatebox{90}{\mancone}}\,}}

\newcommand{\coc}{C_{\rm co}}

\begin{document}

\title{A noncommutative calculus on the cyclic dual of $\Ext$} 

\author{Niels Kowalzig}

\begin{abstract}
We show that if the cochain complex computing $\Ext$ groups (in the category of modules over Hopf algebroids) admits a cocyclic structure,
then the noncommutative Cartan calculus structure on $\Tor$ over $\Ext$ dualises in a cyclic sense to a calculus on $\Coext$ over $\Cotor$. More precisely, the cyclic duals of the chain resp.\ cochain spaces computing the two classical derived functors lead
to complexes that compute the more exotic ones, giving a cyclic opposite module over an operad with multiplication that induce operations such as a Lie derivative, a cap product (or contraction), and a (cyclic) differential, along with higher homotopy operators defining a noncommutative Cartan
calculus up to homotopy.
In particular, this allows to recover the classical Cartan calculus from differential geometry or the Chevalley-Eilenberg calculus for Lie(-Rinehart) algebras without any finiteness conditions or the use of topological tensor products.
\end{abstract}

\address{Dipartimento di Matematica, Universit\`a di Napoli Federico II, Via Cintia, 80126 Napoli, Italy}

\email{niels.kowalzig@unina.it}



\maketitle

\setcounter{tocdepth}{1}
\tableofcontents

\section*{Introduction}

\subsection{Aims and objectives}

Higher structures on cohomology or homology,
such as brackets, products, and differentials, 
are typically only part of a richer structure on pairs of cohomology and homology groups, where one acts on the other in various ways, as a graded module or graded Lie algebra module, for example. Usually, these operations can already be observed on a cochain resp.\ chain level, often encoded in the action of an operad on a module or opposite module, 
fulfilling certain axioms only up to homotopy and accordingly involving more or less explicit higher homotopy operators as well. The probably most basic example here is given by the pair of multivector fields and differential forms, seen as cohomology and homology groups with zero differentials, where the former acts on the latter by contraction and Lie derivative, and both are equipped with differentials that give, depending on the precise context, rise to de Rham or Lie algebra cohomology resp.\ homology. Algebraically, this idea was formalised in \cite{GelDalTsy:OAVONCDG, NesTsy:OTCROAA, TamTsy:NCDCHBVAAFC, Tsy:CH} by the notion of {\em noncommutative differential calculus}, which also runs under the name {\em Batalin-Vilkoviski\u\i\ (BV) module},
and has been an active research topic since \cite{DolTamTsy:FTFHCATA, DolTamTsy:NCATGMC, Lam:VDBDABVSOCYA, Tsy:NCAO, ArmKel:DIOTTTCOAA, Her:HCOKDP, Tam:TTTCOAAALS},  finding its possibly highest degree of abstraction so far in the definition of the {\em Kontsevich-Soibelman} operad (as introduced in \cite{KonSoi:DOAOOATDC, KonSoi:NOAA}, see also \cite[\S4]{DolTamTsy:FOTHCAOHC}) that essentially encodes calculi.
Later work, for example in \cite{KowKra:BVSOEAT}, resulted in a homotopy calculus structure on the cochain and chain complexes that compute $\Ext$ groups and $\Tor$ groups over quite general rings, more precisely over so-called Hopf algebroids, which enlarged the Hochschild case from \cite{NesTsy:OTCROAA} and also allowed for more general coefficients, from which one can deduce, as an example, that the Hochschild cohomology of twisted Calabi-Yau algebras forms a Batalin-Vilkoviski\u\i\ (BV) {\em algebra}. The approach in \cite{KowKra:BVSOEAT} was formalised in an operadic language in \cite{Kow:GABVSOMOO} by determining the minimal ingredients required in order to obtain a (homotopy) noncommutative calculus.

The main objective in the article at hand is to investigate what happens to a (homotopy) noncommutative calculus when applying to it what is called {\em cyclic duality}, which transforms cocyclic objects in cyclic ones and vice versa, see \cite{Con:CCEFE} and \S\ref{boing}. More precisely, by using the operadic formalism developed in \cite{Kow:GABVSOMOO} and the cyclic structure on the cochain complex computing $\Ext$ groups obtained in \cite{Kow:WEIABVA}, we use cyclic duality both on the cochain resp.\ chain complexes that eventually yield the noncommutative calculus on $\Tor$ over $\Ext$ in \cite{KowKra:BVSOEAT} to a obtain a homotopy noncommutative calculus on the chain resp.\ cochain complexes that leads to a natural calculus of $\Coext$ over $\Cotor$ when descending to (co)homology. This approach turns out to be versatile enough to include the classical Cartan calculus in differential geometry as an example.

The pattern behind our construction is quite striking: starting from a cyclic unital opposite module $\cM$ over an operad $\cO$ with multiplication (the chain space that computes $\Tor$ over the cochain space that computes $\Ext$), one obtains a noncommutative calculus on $(H^\bullet(\cO), H_\bullet(\cM))$. Adding the assumption that the operad $\cO$ is cyclic, one can pass to the cyclic duals both for $\cO^\bullet$ and $\cM_\bullet$ with the result that now r\^oles are exchanged and $\cO_\bullet$ is a cyclic unital opposite module over $\cM^\bullet$  (the chain space that computes $\Coext$ over the cochain space that computes $\Cotor$), which means that now $(H^\bullet(\cM), H_\bullet(\cO))$ yields a noncommutative calculus. As a side remark, both $H^\bullet(\cO)$ and $H^\bullet(\cM)$ even become
Batalin-Vilkoviski\u\i\
  {\em algebras} here, that is, a Gerstenhaber algebra whose bracket measures the failure of the cyclic differential to be a (graded) derivation of the cup product. 
   We wonder whether one could observe this sort of dual behaviour on a much more general level only involving, say, two cyclic operads with a mutual action, but were at present not able to make this idea more precise.

   \subsection{Main results}

   In \S\ref{atacvantaggi}, we improve earlier work \cite[Prop.~4.8]{Kow:WEIABVA} by observing that even in the non-finite case the category $ {}_\uhhu \mathbf{aYD}^{\scriptscriptstyle{\rm contra-}\uhhu}$ of anti Yetter-Drinfel'd (aYD) contramodules over a left bialgebroid $(U,A)$, while not being monoidal, is a module category over $\yd$, the monoidal category of Yetter-Drinfel'd (YD) modules over $U$, which relies essentially on the fact that this is already the case for right $U$-contramodules over the monoidal category of left $U$-comodules. Expressed in simpler terms, in Proposition \ref{fliegenervt} we prove that if $M \in  {}_\uhhu \mathbf{aYD}^{\scriptscriptstyle{\rm contra-}\uhhu}$ and $N \in \yd$, then $\Hom_\Aop(M,N)$ is an aYD contramodule over $U$ again. This observation allows to generalise \cite[Cor.~4.13]{Kow:WEIABVA} to more general coefficients in Proposition \ref{arteepolitica}, see the main text for all details:

   \begin{prop}
     If $M$ is an aYD contramodule and $N$ a YD module over a left bialgebroid $(U,A)$ such that $\Hom_\Aopp(N,M)$ is stable, then (when $U_\ract$ is $A$-flat) the cochain complex computing $\Ext^\bullet_U(N,M)$ is a cyclic $k$-module.
  \end{prop}

In a standard way, as briefly explained in Eq.~\eqref{e-mantra}, this yields a degree $-1$ differential on the cochain complex that induces a differential
$
B: \Ext^\bullet_U(N,M) \to \Ext_U^{\bullet-1}(N,M)
$
on cohomology.

One of the main feature of Connes' cyclic category is its self-duality, as mentioned in \S\ref{boing}. This allows to construct, as in Eq.~\eqref{cyclicdual}, from any cocyclic $k$-module a cyclic $k$-module essentially by treating cofaces as degeneracies and codegeneracies as faces, except for one of them the definition of which involves the cocyclic operator (they are infinitely many ways for such a procedure due to the infinite number of autoequivalences of the cyclic category). While it is known that in case the Hochschild cochain complex is cyclic (which, as a side remark, is usually not the case) the result is trivial, in general for Hopf algebroids the situation is richer.
In Lemma \ref{lakritz} and Theorem \ref{salveregina}, we show (under the standing assumption that $U_\ract$ is $A$-flat and $M$ is $A$-injective):

\begin{theorem}
If $(U,A)$ is both a left and a right Hopf algebroid and $M$ a stable aYD contramodule over $U$, then the cyclic dual of the cochain complex that computes $\Ext_U^\bullet(A,M)$ yields a chain complex computing $\Coext_\bullet^U(A,M)$ if $M$ is injective as an $A$-module, along with a degree $+1$ differential $B: \Coext_\bullet^U(A,M) \to \Coext_{\bullet+1}^U(A,M)$. 
  \end{theorem}

More precisely, if $\gamma$ denotes the right $U$-contraaction on $M$ and if we indicate by $u_{[+]} \otimes_A u_{[-]}$ for $u \in U$ a Sweedler-type notation for the inverse of one of the canonical Hopf-Galois maps, that is, the right Hopf structure, we obtain on the chain spaces $C_\bullet(U,M) := \Hom_A(U^{\otimes_A \bullet}, M)$ the following structure maps of a cyclic $k$-module:
\begin{eqnarray*}
(d_i f)(u^1|\ldots|u^{n-1}) \!\!\!\!&=\!\!\!\!& \left\{\!\!\!
\begin{array}{l} 
  \gamma\big(\smap f(\smam |u^1 |\ldots |u^{n-1})\big)
\\ 
 f(u^1 |\ldots |\gD u^i |\ldots | u^{n-1})
\\
f(u^1|\ldots|u^{n-1} |1)
\end{array}\right.  
  \begin{array}{l} \mbox{if} \ i=0, \\ \mbox{if} \
  1 \leq i \leq n-1, \\ \mbox{if} \ i = n,  \end{array} 
  \\[3pt]
 (s_j f)(u^1 | \ldots | u^{n+1}) \!\!\!\! 
  &=\!\!\!\!& f(u^1 | \ldots | \gve(u^{j+1}) | \ldots | u^{n+1}) \ \quad\qquad\mbox{for} \ \, 0 \leq j \leq n,
  \\[3pt]
(t f)(u^1 | \ldots | u^n) \!\!\!\!&=\!\!\!\!& \gamma\Big(((\sma{-} u^1) \mpact
  f)\big(u^2 | \ldots | u^n | 1 \big) \Big),
\end{eqnarray*}
where $\mpact$ denotes the left $U$-action on
$\Hom_A(U^{\otimes_A n}, M)$ as in Eq.~\eqref{mpaction} and the vertical bars denote a certain tensor product over $A$, see Eq.~\eqref{samsung1}. From a broader perspective, this cyclic $k$-module and the corresponding differential $B$ are part of what is called a {\em homotopy noncommutative} or {\em homotopy Cartan calculus}, also known as {\em homotopy BV module}, see
\S\ref{esregnetdurchdieDecke}. Such a differential calculus typically arises from a so-called {\em cyclic opposite module} over an operad with multiplication, as quoted in Theorem \ref{terzamissione}; the operad in question here arises from the complex computing the derived functor $\Cotor^\bullet_U(A,A)$. In this spirit, in Theorem
\ref{plebiscito}, we prove (again with $U_\ract$ flat over $A$ and $M$ injective over $A$):

\begin{theorem}
If $M$ is a stable aYD contramodule over a left bialgebroid $(U,A)$ which is both left and right Hopf, the chain complex computing $\Coext^U_\bullet(A,M)$ can be seen as a cyclic unital opposite module over the cochain complex computing $\Cotor_U^\bullet(A,A)$, seen as an operad with multiplication, such that the underlying cyclic $k$-module structure is the one listed right above.
 \end{theorem}

This, as already mentioned, has Corollary \ref{tarrega}
as an immediate consequence:

 \begin{cor}
   The couple consisting of the cochain complex computing $\Cotor_U^\bullet(A,A)$ and the chain complex computing  $\Coext^U_\bullet(A,M)$ can be equipped with the structure of a homotopy noncommutative calculus if $M$ is a stable aYD contramodule over $U$.
In particular, this induces the structure of a BV module on
$\Coext^U_\bullet(A, M)$
over
$\Cotor^\bullet_U(A,A)$.
 \end{cor}

 Explicitly,
along with the {\em homotopy} or {\em higher $B$-operators} $\cS$ and $\cT$, see Eqs.~\eqref{Regenregenregen},
the calculus operators of cyclic differential, contraction, and Lie derivative 
read as follows:
\begin{small}
   \begin{eqnarray*}
    (B f)(v^0|\ldots|v^{n}) &\!\!\!\!\!
    =&\!\!\!\!\!
    \textstyle\sum\limits^{n+1}_{i=1} (-1)^{(i-1)n} \gamma\Big(\!\smadotp (v^i \mpact f)\big(\smadotm \mancino (v^{i+1} | \ldots | v^{n+1}) | v^1|\ldots|v^{i-1}\big)\!\Big),                    
  \\
(\iota_w f)(v^1|\ldots|v^{n-p}) &\!\!\!\!\!=&\!\!\!\!\! \gamma\big(\smadotp f(\smadotm \mancino (u^1 | \ldots | u^p) | v^1|\ldots|v^{n-p})\big), 
\\[4pt]
    (\cL_w f)(v^1|\ldots|v^{n-p+1})&\!\!\!\!\!=&\!\!\!\!\!
\\
    &&
    \hspace*{-3.38cm}
     \textstyle\sum\limits^{n-p+1}_{i=1} (-1)^{(p-1)(i-1)} 
 f\big(v^1 | \ldots | v^{i-1} | v^i \mancino (u^1 | \ldots |u^p)|
 v^{i+1} | \ldots | v^{n-p+1}\big)
 \\
    &&
  \hspace*{-3.5cm}
  + \textstyle\sum\limits^{p}_{i=1} (-1)^{n(i-1) + p-1}
   \gamma\Big(\!\smadotp (u^i \mpact f)\big(\smadotm \mancino (u^{i+1} | \ldots | u^{p}) |v^1|\ldots|v^{n-p+1}| u^1|\ldots|u^{i-1}\big)\!\Big), 
  \end{eqnarray*}
 \end{small}
for $w := (u^1| \ldots | u^p) \in U^{\otimes_A p}$ 
and $f \in \cM(n)$, where $\mancino$ denotes the diagonal action in the monoidal category $\umod$ of left $U$-modules.

Our main application of this machinery consists in showing in \S\ref{examples} that
the noncommutative calculus on $\Coext$ and $\Cotor$ provides a natural framework for including the classical Cartan calculus in differential geometry as an example: in Theorem \ref{transactions}, we show:

\begin{theorem}
  Let $(A,L)$ be a Lie-Rinehart algebra, where $L$ is projective over $A$ of possibly infinite dimension, and $V\!L$ its universal enveloping algebra. Then the antisymmetrisation map induces an isomorphism of BV modules (or Cartan calculi) between $\big(\!\textstyle\bigwedge^n_A \!L, \Hom_A(\textstyle\bigwedge^n_A \!L, M)\big)$  and $\big(\!\Cotor^\bullet_{V\!L}(A,A), \Coext^{V\!L}_\bullet(A,M)\big)$, where $M$ is an $A$-injective $V\!L$-module.
  \end{theorem}

Here, by isomorphism of BV modules we mean a pair of isomorphisms of the respective underlying $k$-modules that commute with all calculus operators $B, \iota, \cL, \cS, \cT$, and also induce an isomorphism of Gerstenhaber algebras, see Lemma \ref{dumdidum} and Eqs.~\eqref{oskar1a}--\eqref{oskar3} for details.

This, in particular, contains the Chevalley-Eilenberg calculus for Lie algebras and the calculus known for Lie algebroids as vector bundles over smooth manifolds which, in turn, includes the classical Cartan calculus if the vector bundle in question is the tangent bundle.

A related but more restrictive result was already obtained in \cite{KowKra:BVSOEAT}
by developing a calculus on $\Tor$ over $\Ext$. There, however, finiteness of $L$ as an $A$-module was necessary to be assumed since the construction not only hinges on the jet space $J\!L$ as a bialgebroid dual to $V\!L$ but also passes through a sort of double dual that plays the r\^ole of the space of multivector fields; in particular, one has to make use of topological tensor products along with completions. Here, none of all this is required and the result can be obtained by purely algebraic operations. Finally, it appears (to us) more natural to regard $V\!L$ as the space of differential operators on a manifold (in case $L$ arises from a Lie algebroid) instead of $\Hom_A(J\!L, A)$.

\subsection{Notation and conventions}
\label{schonweniger}
All notation for bialgebroids, cyclic modules, operads etc.\ is explained in the respective sections or appendices at the end of the main text. Here, we only introduce some basic notation globally used.

The symbol $k$ always denotes a commutative ring, usually of characteristic zero.
For a left bialgebroid $(U,A)$ and a left $U$-module $M$, we most of the time denote the $U$-action just by juxtaposition, except for a few cases: for example, the monoidal structure on the category $\umod$ of left $U$-modules is reflected by the diagonal $U$-action on the tensor product $N \otimes_A M$ of two left $U$-modules $N, M$, that is,
\begin{equation}
  \label{mancino}
u \mancino (n \otimes_A m) :=
\gD(u)(n \otimes_A m) =
u_{(1)} n \otimes_A u_{(2)} m
\end{equation}
for $n \in N, m \in M$, and $u \in U$. If $U$ is on top a left resp.\ right Hopf algebroid
(see \S\ref{bialgebroids}),
one obtains a left $U$-module structure on $\Hom_\Aopp(N,M)$ resp.\ on $\Hom_A(N, M)$: in the first case, for all $f  \in \Hom_\Aopp(N,M)$, set
\begin{equation}
  \label{pmaction}
(u \pmact f)(n) := u_+ f(u_- n), 
  \end{equation}
and in the second case, for all $g \in \Hom_A(N,M)$, put
\begin{equation}
    \label{mpaction}
(u \mpact g)(n) := u_{[+]} g(u_{[-]} n),
\end{equation}
see right below Eqs.~\eqref{nochmehrRegen} or the beginning of Appendix \ref{nebula} for the notation used here.
Recall from Eq.~\eqref{pergolesi} the various triangle notations $\lact, \ract, \blact, \bract$ that denote the four $A$-module structures on the total space $U$ of a bialgebroid, and occasionally even on a $U$-module. We abbreviate tensor products $U_\ract \otimes_A \due U \lact {}$ with a vertical bar and tensor products in $ \due U \blact {} \otimes_\Aopp U_\ract$ with a comma, that is, write
\begin{equation}
  \label{samsung1}
  (u^1| \ldots|u^n) := u^1 \otimes_A \cdots \otimes_A u^n \in {{}_\lact U_\ract}^{\otimes_A n},
\end{equation}
as well as
\begin{equation}
  \label{samsung2}
(u^1, \ldots,u^n) := u^1 \otimes_\Aopp \cdots \otimes_\Aopp u^n \in {{}_\blact U_\ract}^{\otimes_\Aopp n}.
\end{equation}
This would somehow make more sense the other way round as the analogue of the {\em bar} resolution is defined on $U^{\otimes_\Aopp n}$, while $U^{\otimes_A n}$ is the right space for the {\em cobar} resolution, but for notational consistency with the predecessor \cite{Kow:WEIABVA} of this article, we decided to stick to the comma notation with respect to the tensor powers over $\Aop$.

Finally, to keep things simple in homological considerations, we always assume (and sometimes even repeat this explicitly) that $\due U {} \ract$ is flat as an $A$-module.

\section{Lie-Rinehart algebras and Hopf algebroids}

\subsection{Lie algebroids and Lie-Rinehart algebras}

The notion of {\em Lie algebroids} 
starts by the simple observation that smooth sections of the tangent bundle $TQ$ over a smooth manifold $Q$, {\em i.e.}, the space of vector fields not only constitutes a Lie algebra over the ground ring $\C$ or $\R$ but also is a module over the smooth functions $\cC^\infty(Q)$ and in turn acts on these by the Lie derivative. Motivated by this example, in a Lie algebroid the tangent bundle is typically replaced by a sort of {\em alternative} tangent bundle, that is, a vector bundle $E \to Q$ endowed with a Lie algebra morphism $E \to TQ$, which has its applications coming from, {\em e.g.}, connection theory or Poisson geometry; see \cite{Pra:TDLPLGD} for one of the original references and \cite{CanWei:GMFNCA} for a concise overview.

The concept of {\em Lie-Rinehart algebra} as introduced by \cite{Rin:DFOGCA} and elaborated on in \cite{Hue:PCAQ} is the algebraic version of this construction. More precisely,
a Lie-Rinehart algebra $(A,L)$ is a pair consisting of a commutative (usually unital) $k$-algebra $A$ and an $A$-module $L$ which is simultaneously a $k$-Lie algebra equipped with a morphism of $k$-Lie algebras
$L\to\Der_k A, \ X \mapsto \{a \mapsto X(a)\}$, the {\em anchor} map, 
which is required to be a morphism of $A$-modules
such that for all $ X,Y \in L, \ a \in A$ the Leibniz identity
 $ {[X,aY]} = a[X,Y]+ X(a)Y $     
holds.

\subsection{Examples}
Obvious examples of Lie-Rinehart algebras from the algebraic side arise both from Lie algebras in which the base algebra equals the ground ring and the action on it is trivial, or from Lie algebroids from the geometric side.  One can also construct Lie algebroids as an infinitesimal version of Lie groupoids analogous to as Lie algebras are associated to Lie groups (see again \cite{CanWei:GMFNCA} for all technical difficulties and differences to the Lie group case). More advanced is the construction of the {\em Atiyah sequence} \cite{Ati:CACIFB} which associates a Lie algebroid to the {\em gauge groupoid} of a principal bundle used to study complex analytic connections.

\subsection{The universal enveloping algebra of differential operators}
Much the same way as one associates a universal algebra to any Lie algebra, for any Lie-Rinehart algebra $(A,L)$ there is a universal algebra $V\!L$, see \cite{Rin:DFOGCA, Hue:PCAQ}. We do not need (nor want) to discuss all technical details of its construction here but only mention that
the universal object $V\!L$
comes along with two canonical morphisms $A \to V\!L$ resp.\ $L \to V\!L$ of $k$-algebras resp.\ $k$-Lie algebras, the first one being always injective, the second one in general if $L$ is $A$-projective, in which case (similar as for Lie algebras) one proves a Poincar\'e-Birkhoff-Witt theorem that states $S_A L \simeq \mathrm{gr}(V\!L)$, see again {\em op.~cit.} for all subtleties involved.

If the Lie-Rinehart algebra $(A,L)$ arises from a Lie algebroid over a smooth manifold, one might want to consider $V\!L$ as the algebra of {\em differential operators} on that manifold.

\subsection{Bialgebroids {\cite{Tak:GOAOAA}}}
\label{bialgebroids}
A left bialgebroid $(U, A, \gD, \gve, s, t)$ is a generalisation of a $k$-bialgebra over a noncommutative base ring $A$; more precisely, it consists of a compatible algebra and coalgebra structure over $\Ae$ resp.\ over $A$; see, for example, \cite{Boe:HA} for details. In particular, it comes along with a ring homomorphism resp.\ antihomomorphism $s,t: A \to U$ (called source resp.\ target) that equip $U$ with four commuting $A$-module structures, denoted
\begin{equation}
  \label{pergolesi}
  a \blact b \lact u \ract c \bract d := t(c)s(b)us(d)t(a)
\end{equation}
  for $u \in U, \, a,b,c,d \in A$, and this situation will be abbreviated by
$
\due U {\blact \lact} {\ract \bract}
$
or any variation thereof, depending  on the action considered in a specific construction. In the same spirit, there is an obvious forgetful functor $\umod \to \amoda$ and therefore, for a left $U$-module $M$, we sometimes denote the induced $A$-bimodule structure by $a \lact m \ract b := s(a)t(b)m$ for $m \in M$, $a, b \in A$. Furthermore, as mentioned, along with a product in $U$, one has a 
coproduct
$$
\Delta: U \to \due U {} \ract \times_A \due U \lact {} \subset \due U {} \ract \otimes_A \due U \lact {}, \quad u \mapsto u_{(1)} \otimes_A u_{(2)}
$$
and a counit $\gve: U \to A$ subject to certain technicalities which we are not going to explain here but refer to \cite{Tak:GOAOAA} or elsewhere. Here,
$$
U \times_{\scriptscriptstyle A} U   :=
   \big\{ {\textstyle \sum_i} u_i \otimes  v_i \in U_{\!\ract}  \otimes_\ahha \!  \due U \lact {} \mid {\textstyle \sum_i} a \blact u_i \otimes v_i = {\textstyle \sum_i} u_i \otimes v_i \bract a,  \ \forall a \in A  \big\}
$$
is sometimes called {\em Sweedler-Takeuchi product}.
   
\subsection{Left and right Hopf algebroids {\cite{Schau:DADOQGHA}}}
\label{schau}
Generalising Hopf algebras (bialgebras with an antipode) to noncommutative base rings is {\em much} less straightforward and instead of asking for an antipode to exist, one rather wants a certain Hopf-Galois map to be invertible. More precisely, for a left bialgebroid $(U, A)$,  one considers the $U$-module morphisms
\begin{equation}
  \label{nochmehrRegen}
\begin{array}{rclrcl}
\ga_\ell : \due U \blact {} \otimes_{\Aopp} U_\ract &\to& U_\ract  \otimes_\ahha  \due U \lact,
& u \otimes_\Aopp v  &\mapsto&  u_{(1)} \otimes_\ahha u_{(2)}  v, \\
\ga_r : U_{\!\bract}  \otimes_\ahha \! \due U \lact {}  &\to& U_{\!\ract}  \otimes_\ahha  \due U \lact,
&  u \otimes_\ahha v  &\mapsto&  u_{(1)}  v \otimes_\ahha u_{(2)},
\end{array}
\end{equation}
and calls the left bialgebroid $(U,A)$ a {\em left Hopf algebroid} if $ \alpha_\ell  $ is a bijection and a {\em right Hopf algebroid}
if $\ga_r$ is so. For later convenience in explicit computations, the shorthand notation
$
u_+ \otimes_\Aopp u_-  :=  \alpha_\ell^{-1}(u \otimes_\ahha 1)
$
and
$
   u_{[+]} \otimes_\ahha u_{[-]}  :=  \alpha_r^{-1}(1 \otimes_\ahha u)
$
   will prove handy and these maps are sometimes called {\em translation maps}. In Appendix \S\ref{nebula}, we collect all necessary compatibility identities that hold between these two structures and the coproduct.
   Let us stress at this point that
   a left bialgebroid which is both left and right Hopf still does not imply the existence of an antipode required in the definition of a (full) Hopf algebroid in \cite{Boe:HA}; a concrete counterexample is given by the universal enveloping algebra $V\!L$ as we are going to discuss next.

   \subsection{The bialgebroid of differential operators}
   \label{voila}
   The universal enveloping algebra $V\!L$ 
   of a Lie-Rinehart algebra $(A,L)$ is not only a left bialgebroid but also a left and right Hopf algebroid over this left bialgebroid structure; this still does not give a (full) Hopf algebroid in the sense of \cite{Boe:HA} as in general an antipode does not exist \cite[Prop.~3.11]{KowPos:TCTOHA}. The algebra $V\!L$ is generated by elements $a \in A$ and $X \in L$, and the left and right Hopf algebroid structure on $V\!L$ comes out as follows: source and target maps are equal and equal the canonical injection $A \to V\!L$, henceforth suppressed from notation. We therefore modify the notation for the tensor products in the Hopf-Galois maps \eqref{nochmehrRegen} by indicating the position of the elements in $A$ in the quotient, that is, write $V\!L \otimes^{ll} V\!L := V\!L_\ract  \otimes_\ahha  \due {V\!L} \lact {}$ and  $V\!L \otimes^{lr} V\!L := \due {V\!L} \blact {} \otimes_{\Aop} V\!L_\ract$, which in this case coincides with $V\!L_{\bract}  \otimes_A \! \due {V\!L} \lact {}$. On generators, the structure maps then read as
\begin{small}
\begin{equation}
\label{regalate}
  \begin{array}{rclrcl}
    \gD(X) &\!\!\!\!\!=&\!\!\!\!\! X \otimes^{ll} 1 + 1 \otimes^{ll} X, &
    X_+ \otimes^{rl} X_- &\!\!\!\!\!=&\!\!\!\!\!  X_{[+]} \otimes^{rl} X_{[-]} =   X \otimes^{rl} 1 - 1 \otimes^{rl} X,
    \\
    \gD(a) &\!\!\!\!\!=&\!\!\!\!\! a \otimes^{ll} 1, 
& a_+ \otimes^{rl} a_- &\!\!\!\!\!=&\!\!\!\!\!  a_{[+]} \otimes^{rl} a_{[-]} =  a \otimes^{rl} 1, 
  \end{array}
  \end{equation}
\end{small}
along with $\gve(X) = 0$ and $\gve(a) = a$, and it is a quick check that this indeed defines a left bialgebroid and that the two Hopf-Galois maps indeed ``invert'' the coproduct in the sense described above.
The bialgebroid $V\!L$ is, in particular, cocommutative and hence the considerations made for this special case in \S\ref{nirvana} will apply. Again, if the Lie-Rinehart algebra $(A,L)$ arises from a Lie algebroid over a smooth manifold, we might want to call this the {\em bialgebroid of differential operators} on that manifold.


\section{\for{toc}{The HKR map and the derived functors $\Cotor$ and $\Coext$}\except{toc}{The Hochschild-Kostant-Rosenberg map and the derived functors $\Cotor$ and $\Coext$}}
\label{urgent}

\subsection{The Hochschild-Kostant-Rosenberg map}
The
HKR
map of
antisymmetrisation \cite{HocKosRos:DFORAA} has been, in a sense, the starting point of Kontsevich's formality considerations \cite{Kon:DQOPM} and correspondingly appears in various contexts dealing with related questions, see, for example, \cite{Dol:CAEFT, Cal:FFLA}.
The form in which we need it here is as follows. Assume from this section onwards
$\Q \subseteq k$ and consider the antisymmetrisation map
\begin{equation}
   \begin{array}{rcl}
   \label{hkr}
   \mathrm{Alt}:  \textstyle\bigwedge_A^n \!L &\to& V\!L^{\otimes^{ll}_A n} 
 \\[2pt]
   X^1 \wedge \cdots \wedge X^n &\mapsto&
   {1}/{n!} \sum_{\gs \in S(n)} (-1)^\gs (X^{\gs(1)}|\ldots|X^{\gs(n)}). 
 \end{array}
   \end{equation}
Before we can state its homological (or homotopical) properties, we need to introduce the right receptacles for a (co)homology theory for $V \!L$ or, more general, for any left bialgebroid $(U, A)$, which will lead us to the somewhat exotic derived functors $\Cotor$ and $\Coext$. These are, in a sense we will
briefly explain now, dual to the well-known $\Tor$ and $\Ext$ and might appear at first glance a possibly not too urgent extension of the theory but as we are going to see not only in the example section \S\ref{examples}, they yield a direct algebraic and natural approach if one wants to embed the Cartan calculus in differential geometry into a more abstract framework.

\subsection{Cotor and comodules}
\label{cotor}
In this subsection, we will describe the derived functor of the cotensor product, which is called $\Cotor$ in analogy to the derived functor $\Tor$ of the ordinary tensor product; {\em cf.}~\cite{EilMoo:FORHA, Doi:HC} for classical information on the subject in the realm of customary coalgebras, \cite{BrzWis:CAC} for general corings, or still \cite[App.~A]{Rav:CCASHGOS} for commutative bialgebroids. For comodules over bialgebroids and involved technical features see \cite{Tak:GOAOAA}.

For a general bialgebroid $(U,A)$, the categories $\ucomod$ and $\comodu$ of left resp.\ right comodules
are not necessarily abelian, but are so if we assume that $\due U \lact {}$ resp.\ $\due U {} \ract$ are flat over $A$; hence, let us directly assume that $U_\ract$ is flat, as mentioned in \S\ref{schonweniger}. To shorten terminology, we shall not use a ``relative'' language, that is, we call a right $U$-comodule $P$ {\em injective} if it is a direct summand in a {\em free} one, that is, a comodule of the form $X \otimes_A U$  for a right $A$-module $X$.

\begin{dfn}
  Let $(U,A)$ be a left bialgebroid, $P \in \comodu$ with right coaction $\gvr_P$,
  and $M \in \ucomod$ with left coaction $\gl_M$. 
  The {\em cotensor product} $P \bx_U M$ is defined as the equaliser of the pair of maps
\begin{equation*}
  \begin{split}
  (\gvr_P \otimes_A M, P \otimes_A \gl_M): P \otimes_A M \rightrightarrows P \otimes_A \due U \lact \ract \otimes_A M,
  \end{split}
\end{equation*}
that is, as the kernel of the difference map.
\end{dfn}
More explicitly, the cotensor product is given by the subspace 
$$
P \bx_U M = \big\{p \otimes_A m \in P \otimes_A M \mid p_{(0)} \otimes_A p_{(1)} \otimes_A m =  p \otimes_A m_{(-1)} \otimes_A m_{(0)} \},
$$
where we wrote $\gvr_P(p) = p_{(0)} \otimes_A p_{(1)}$ and $\gl_M(m) = m_{(-1)} \otimes_A m_{(0)}$ for the right resp.\ left $U$-coaction. For any $M \in \ucomod$, there is a natural isomorphism $U \bx_U M \to M$ given by $u \otimes_A m \mapsto \gve(u)m$ with inverse the left coaction $\gl_M$.
More generally, for any right $A$-module $X$ (that is, for any free right $U$-comodule of the form $X \otimes_A U)$, we have an isomorphism
\begin{eqnarray}
  \label{palazzofalconieri}
\phi: (X \otimes_A U) \bx_U M \to X \otimes_A M, \quad (x \otimes_A u) \bx_U m \mapsto x \otimes_A \gve(u)m, 
\end{eqnarray}
with inverse induced by the coaction of $M$ as above.
The functor of taking cotensor products is left exact in the first variable if $M$ is flat as a left $A$-module; the same holds in the second variable if $P$ is flat as a right $A$-module. As a consequence, we can define its right derived functors $\Cotor_U$: more precisely, considering that under the flatness assumptions on $U$ the category $\comodu$ is abelian has enough injectives, any
resolution $P \to I^\bullet$ of the right $U$-comodule $P$ by a 
cochain complex $(I^\bullet, \partial')$ of injective right $U$-comodules
is acyclic for the functor $- \bx_U M$, and we therefore define
$$
\Cotor^\bullet_U(P,M) := H\big(I^\bullet \bx_U M, \partial' \otimes_A M\big).
$$
The standard way of resolving the right $U$-comodule $P$ is by the well-known cobar cochain complex: set $\Cob^n(P, U) := P \otimes_A {\due U \lact \ract}^{\otimes_A n+1}$ for any $n \in \N$, and define the 
differential
$\partial' = \sum^{n+1}_{i=0} (-1)^i \partial'_i: \Cob^n(P, U)  \to \Cob^{n+1}(P, U)$,
where
\begin{small}
 \begin{equation}
\label{landliebecobar}
   \begin{array}{rcl}
 \partial'_i(p| u^1| \ldots| u^{n+1}) \!\!\!\!
  &=\!\!\!\!&
 \left\{
 \!\!\!
 \begin{array}{l} 
 (\gvr_P(p)|u^1|\ldots|u^{n+1})
 \\ 
 (p|u^1|\ldots|\gD(u^i)| \ldots | u^{n+1})
 \end{array}\right.  
   \begin{array}{l} \mbox{if} \ i=0, \\ \mbox{if} \
     1 \leq i \leq n+1,
   \end{array}
 \end{array}
 \end{equation}
\end{small}
%
%
using the notation introduced in \eqref{samsung1}. As a consequence, $\Cotor_U^\bullet(P,M)$ can be computed by the cochain complex $\Cob^\bullet(P, U) \bx_U M$ with differential $\partial'$. 
Now, the right coaction on $\Cob^n(P, U)$ is simply defined by the coproduct on the rightmost tensor factor of $U$, which therefore yields a
free (hence
injective) resolution of $P$.
Applying the isomorphism $\phi$ from \eqref{palazzofalconieri}, we see that $\Cotor_U^\bullet(P,M)$ can effectively be computed by the chain complex $P \otimes_A U^{\otimes_A \bullet} \otimes_A M$ with differential $\partial  = \sum^{n+1}_{i=0} (-1)^i \partial_i$ in degree $n$, which is typically more convenient to consider.
Here, the cofaces $\partial_i := \phi \circ \partial'_i \circ \phi^{-1}$ come out as:
\begin{small}
  \begin{equation}
    \label{landliebecotor}
 \begin{array}{rcl}
 \partial_i(p| u^1| \ldots| u^n |m) \!\!\!\!
  &=\!\!\!\!&
 \left\{
 \!\!\!
 \begin{array}{l} 
 (\gvr_P(p)|u^1|\ldots|u^n|m)
 \\ 
 (p|u^1|\ldots|\gD(u^i)| \ldots | u^n|m)
 \\
 (p|u^1|\ldots|u^n|\gl_M(m))
 \end{array}\right.  
   \begin{array}{l} \mbox{if} \ i=0, \\ \mbox{if} \
     1 \leq i \leq n,
     \\ \mbox{if} \ i = n+1.
   \end{array}
 \end{array}
 \end{equation}
\end{small}
In case $P=A$, we will denote the resulting chain complex by
\begin{equation}
  \label{daa}
\coc^\bullet(U,M) := U^{\otimes_A \bullet} \otimes_A M,
\end{equation}
with differential $\partial$ as above and right $U$-coaction on $A$ given by the target map. With this notation, we have
$$
\Cotor_U^\bullet(A,M) = H(\coc^\bullet(U,M), \partial).
$$

\subsection{Coext and contramodules}
\label{coext}
In this subsection, another not too well-known derived functor is introduced, the so-called $\Coext$, the definition of which for coalgebras over a commutative ring appears in \cite[\S0.2]{Pos:HAOSAS} and possibly (much) earlier elsewhere. Again, we have to adapt the construction
given in {\em op.~cit.}\ to a  relative setting as we are dealing with corings over a (possibly noncommutative) base algebra $A$. As the construction amounts to a derived functor of the functor of {\em cohomomorphisms} which map comodules to {\em contramodules}, let us introduce these (equally essentially unknown) objects first.

\subsubsection{Contramodules over bialgebroids}
Contramodules over coalgebras were introduced in \cite{EilMoo:FORHA}
not too long after the notion of comodules but are, in striking contrast to the latter, basically unknown to most of the mathematical community.
They are dealt with, for example, in \cite{BoeBrzWis:MACOMC, Brz:HCHWCC} and gained the attention they deserve in particular in \cite{Pos:HAOSAS, Pos:C}.
For finite dimensional bialgebras (or bialgebroids), a contramodule should be thought of as a module over the dual.
%
They also pop up as natural coefficients in the cyclic theory of $\Ext$ groups and
were implicitly used in the classical cyclic cohomology theory by Connes \cite{Con:NCDG} by choosing coefficients in the linear dual of an algebra, as explained in \cite[\S6]{Kow:WEIABVA}.

\begin{dfn}
\label{schoenwaers}
A {\em right contramodule} over a left bialgebroid $(U,A)$ is a right $A$-module $M$ along with a right $A$-module map 
$$
\gamma: \Hom_\Aopp(U_\ract,M) \to M, 
$$
called the {\em contraaction}, subject to {\em contraassociativity},
\begin{small}
  \begin{equation*}
  \begin{split}
  &	\xymatrix{\Hom_\Aopp(U, \Hom_\Aopp(U,M))
	\ar[rrr]^-{\scriptstyle{\Hom_\Aopp(U,\gamma)}} \ar[d]_-{\simeq}&
	& &
	\Hom_\Aopp(U,M) \ar[d]^-{\gamma} \\
	\Hom_\Aopp(U_\ract \otimes_\ahha \due U \lact {},M)
	\ar[rr]_-{\scriptstyle{\Hom_\Aopp(\gD_\ell, M)}} && \Hom_\Aopp(U,M) \ar[r]_-{\gamma}
&	M, }
\\
& \!\!\!\!\!\!\!\!\!\!\!
\mbox{\normalsize{as well as {\em contraunitality},}}
\\
& \xymatrix{\Hom_\Aopp(A,M) \ar[rr]^-{\Hom_\Aopp(\gve,M)} \ar[drr]_-{\simeq} & & \Hom_\Aopp(U,M)
    \ar[d]^-{\gamma} \\ & & M.  }
\end{split}
\end{equation*}
\end{small}
\end{dfn}
The adjunction of the leftmost vertical arrow in the first diagram is to be understood with respect to the right $A$-module structure on $\Hom_\Aopp(U_\ract,M)$ defined by 
$
fa := f(a \lact -) 
$
for $a \in A$; the right $A$-linearity of $\gamma$ in the definition then means
\begin{equation}
\label{passionant}
\gamma\big(f(a \lact -) \big) = \gamma(f)a.
\end{equation}
Observe that there also is an induced left $A$-action on $M$ given by
\begin{equation}
\label{alleskleber}
am := \gamma\big(m\gve(- \bract a) \big) = \gamma\big(m\gve(a \blact -) \big),
\end{equation}
turning $M$ into an $A$-bimodule, and with respect to which $\gamma$ becomes an $A$-bimodule map:
\begin{equation}
  \label{tamtamdatam}
\gamma\big(f ( - \bract a)\big) = a \gamma\big(f(-)\big),
\end{equation}
as shown in \cite[Eq.~(2.37)]{Kow:WEIABVA}. In particular, we obtain
a forgetful functor 
\begin{equation}
\label{gaeta}
\mathbf{Contramod}\mbox{-}U \to \amoda
\end{equation}
from the category of right $U$-contramodules to that of $A$-bimodules.

In general, we denote the ``free entry'' in the contraaction $\gamma$ by hyphens or dots: for $f \in \Hom_\Aopp(U,M)$ we may write both $\gamma(f \sma{-})$ as well as $\gamma(f \sma{\cdot})$ or simply $\gamma(f)$, depending on readability in long computations: this way, the contraassociativity may be compactly expressed as
\begin{equation}
\label{carrefour1}
\dot\gamma\big(\ddot\gamma(g( \cdot \otimes_\ahha \cdot\cdot))\big) 
= \gamma\big(g(-_{(1)} \otimes_\ahha -_{(2)})\big),
\end{equation}
for $g \in \Hom_\Aopp(U_\ract \otimes_\ahha \due U \lact {},M)$,
where the dots match the map $\gamma$ with the respective argument, and
contraunitality as 
\begin{equation}
\label{carrefour2}
\gamma( m \gve\sma{-}) = m
\end{equation}
for $m \in M$. Finally, do not confuse the operation of contraaction with that of {\em contraction} dealt with from \S\ref{bononiae} onwards.

\begin{example}
  \label{trivial}
  Eq.~\eqref{passionant} in general excludes the existence of a {\em trivial} right contraaction $f \mapsto f(1)$, in full analogy to the fact that for bialgebroids in general there is no trivial (left or right) coaction. However, if $A$ is commutative and source and target map happen to coincide, then such a trivial contraaction is possible. This is, for example, the case for bialgebras or cocommutative bialgebroids, which we will explicitly exploit in \S\ref{nirvana} resp.~\S\ref{examples}. In particular, the universal enveloping algebra $V\!L$ of a Lie-Rinehart algebra $(A,L)$ is a cocommutative left bialgebroid and hence, any right (hence left) $A$-module $M$ can be given a right $V\!L$-contramodule structure by means of
  \begin{equation}
    \label{shaundasschaf1}
\Hom_A(V\!L, M) \to M, \quad f \mapsto f(1),
  \end{equation}
  keeping in mind that here $A = \Aop$.
  \end{example}

\begin{example}
  \label{bienenapotheke}
Essentially for the same reason, again in contrast to coalgebra theory, dualising bialgebroid comodules generally does {\em not} furnish examples of contramodules:
  in case of a coalgebra $C$, that is, for $A = k$, and $N$ a left $C$-comodule with coaction $\gl_N$, the linear dual $\Hom_k(N,k)$ is a right $C$-contramodule with contraaction
  $\gamma := \Hom_k(\gl_N, k)$. Trying to generalise this to a bialgebroid $(U,A)$, for $N \in \ucomod$ neither the right dual $\Hom_\Aop(N, A)$ nor the left dual  $\Hom_A(N, A)$ make this formula well-defined since linearity of the left coaction reads
  $\gl_N(anb) = a \lact n_{(-1)} \bract b \otimes_A n_{(0)}$
  for $a,b \in A$ and $n \in N$.
  \end{example}

\subsubsection{Cohomomorphisms over bialgebroids} We now have the necessary ingredients to deal with the space of cohomomorphisms and its derived functors.

\begin{dfn}
  Let $(U,A)$ be a left bialgebroid, $(P, \gvr_P)$ a right $U$-comodule 
  and $(M, \gamma)$ a right $U$-contramodule.
  The {\em space of cohomomorphisms} $\Cohom_U(P,M)$ is defined as the coequaliser of the pair of maps
\begin{equation*}
  \begin{split}
 & \big(\Hom_\Aopp(\gvr_P, M) , \Hom_\Aopp(P, \gamma)\big):
\\
 & \qquad \Hom_\Aopp(P \otimes_A \due U \lact {}, M) \simeq \Hom_\Aopp(P, \Hom_\Aopp(U,M)) \rightrightarrows \Hom_\Aopp(P,M), 
  \end{split}
  \end{equation*}
that is, the cokernel of the difference map.
\end{dfn}
Here, the $\Aop$-linearity on both sides of the adjunction refers to $\due U {} \ract$.
More explicitly, the space of cohomomorphisms can be described as the quotient
$$
\Cohom_U(P,M) = \Hom_\Aopp(P,M)/I,
$$
where $I$ is the $k$-module generated by
$$
\big\{g \circ \gvr_P - \dot\gamma(g(- \otimes_A \cdot)) \mid g \in \Hom_\Aopp(P \otimes_A \due U \lact {}, M) \big\}.
$$
That this yields a well-defined construction with respect to the right $A$-action follows from
\begin{equation*}
  \begin{split}
g(\gvr_P(pa)) - \gamma\big(g(pa \otimes_A \cdot)\big)
&= g(p_{(0)} \otimes_A p_{(1)} \ract a) -  \gamma\big(g(p \otimes_A a \lact \sma{\cdot})\big)
\\
&=
 g(p_{(0)} \otimes_A p_{(1)})a -  \gamma\big(g(p \otimes_A \cdot)\big)a
  \end{split}
  \end{equation*}
for $a \in A, p \in P$, using right linearity of the right coaction $\gvr_P : p \mapsto p_{(0)} \otimes_A p_{(1)}$ along with \eqref{passionant}. 

For any right $U$-contramodule $M$, there is a natural isomorphism
$
\Cohom_U(U,M) \to M, \ f \mapsto \gamma(f)
$
with inverse $m \mapsto m\gve(\cdot)$. More generally, for any right $A$-module $X$ (that is, for any 
free right $U$-comodule of the form $X \otimes_A U$), we have an isomorphism
\begin{eqnarray}
  \nonumber
\gvt: \Cohom_U(X \otimes_A U,M) &\to& \Hom_\Aopp(X,M), 
  \\
  \label{gethsemane1}
  f &\mapsto& \big\{ x \mapsto \gamma\big(f(x \otimes_A \cdot)\big)  \big\},
  \\
    \label{gethsemane2}
 \big\{  g(x\gve(u)) \mapsfrom x \otimes_A u  \big\} &\mapsfrom& g.
  \end{eqnarray}
Similarly as for coalgebras \cite[\S0.2.5]{Pos:HAOSAS}, the functor
$\Cohom_U(-,M)$
over a bialgebroid $(U,A)$ is right exact if
$U_\ract$ is $A$-flat 
and $M$ injective as a (right) $A$-module, and hence we can define in a standard way its left derived functors $\Coext^U$: similarly to the preceding subsection, any
resolution $P \to I^\bullet$ of the right $U$-comodule $P$ by a 
cochain complex $(I^\bullet, \partial')$ of injective right $U$-comodules
is acyclic for the functor
$
\Cohom_U(-, M)
$,
and we therefore define
$$
\Coext^U_\bullet(P, M) := H\big(\!\Cohom_U(I^\bullet, M), \Hom_\Aopp(\partial', M)\big).
$$
Using the cobar cochain complex from \eqref{landliebecobar} again, $\Coext^U_\bullet(P,M)$ can be computed by the chain complex $\Cohom_U(\Cob^\bullet(P, U), M)$ with differential $b' = \sum^{n+1}_{i=0} (-1)^i b'_i$, where $b'_i f := f \circ \partial'_i$ for any $f \in \Cohom_U(\Cob^n(P, U), M)$.
Again, considering the comodule structure of the cobar complex and applying the isomorphism $\gvt$ from \eqref{gethsemane1}, this time we see that $\Coext^U_\bullet(P,M)$ can effectively be computed by the chain complex $\Hom_\Aopp(P \otimes_A U^{\otimes_A \bullet},M)$ with differential $b  = \sum^{n}_{i=0} (-1)^i b_i$ in degree $n \in \N$, which usually is more practical, again.
Here, the faces $b_i := \gvt \circ b'_i \circ \gvt^{-1}$ by a quick computation using Eqs.~\eqref{gethsemane1}, \eqref{gethsemane2}, \eqref{passionant}, and \eqref{carrefour2}, result into 
\begin{small}
  \begin{equation}
    \label{landliebekirsch}
 \begin{array}{rcl}
 (b_i f)(p| u^1| \ldots| u^{n-1}) \!\!\!\!
  &=\!\!\!\!&
 \left\{
 \!\!\!
 \begin{array}{l} 
 f(\gvr_P(p)|u^1|\ldots|u^{n-1})
 \\ 
 f(p|u^1|\ldots|\gD(u^i)| \ldots | u^{n-1})
 \\
 \gamma(f(p|u^1|\ldots|u^{n-1}|\sma{\cdot}))
 \end{array}\right.  
   \begin{array}{l} \mbox{if} \ i=0, \\ \mbox{if} \
     1 \leq i \leq n-1,
    \\ \mbox{if} \ i = n,
   \end{array}
 \end{array}
 \end{equation}
\end{small}
for any $f \in \Hom_\Aopp(P \otimes_A U^{\otimes_A n},M)$. In case $P=A$, we will denote the resulting chain complex as
\begin{equation}
  \label{dee}
D_\bullet(U,M) := \Hom_\Aopp(U^{\otimes_A \bullet},M)
\end{equation}
with differential $b$ as above and right $U$-coaction on $A$ again given by the target map.
With this notation, we have
$$
\Coext^U_\bullet(A,M) = H(D_\bullet(U,M), b).
$$

Of course, one could equally resolve $M$ by (relative) projective contramodules (see \cite[\S0.2]{Pos:HAOSAS} again) to compute $\Coext^U_\bullet(P,M)$ but we are not going to pursue this possibility here.

\subsection{The HKR and its pull-back map as quasi-isomorphisms}

Having gathered all required material on a general level, we can now come back to the homological and homotopical properties of the HKR map.

 \begin{lemma}
   \label{dumdidum}
   The HKR map is a quasi-isomorphism of cochain complexes which induces an isomorphism
   $
   \textstyle\bigwedge_A^\bullet \!L \simeq \Cotor_{V\!L}^\bullet(A,A)
   $
   of Gerstenhaber algebras, with $L$ assumed to be $A$-flat. For a left $V\!L$-module $M$, seen as a trivial right $V\!L$-contra\-module,  
   the pull-back $\Hom_A({\rm Alt},M)$ yields a quasi-isomorphism
   \begin{equation}
     \label{herbst}
   \big(D_\bullet(V\!L,M), b\big) \lra  \big(\!\Hom_A(\textstyle\bigwedge^\bullet_A \!L, M), 0\big)
   \end{equation}
   of chain complexes. In particular, assuming $M$ to be $A$-injective, on homology
   \begin{equation}
     \label{erkaeltetmalwieder}
 \Coext^{V\!L}_\bullet(A,M) \stackrel{\simeq}{\lra} \Hom_A(\textstyle\bigwedge^\bullet_A \!L, M)
  \end{equation}
 holds. 
 \end{lemma}

The first statement of the above lemma is well-known in its various guises, see, for example, \cite[Thm.~3.13]{Kow:BVASOCAPB}
for the statement in precisely the same context as here (and Corollary 3.8 in {\em op.~cit.}~again, which proves that the $\Cotor$ groups over a general bialgebroid $(U, A)$ constitute a Gerstenhaber algebra). The second statement simply follows from the comments that will have been made at the beginning of \S\ref{nirvana} along with the observation in Eq.~\eqref{wasserbombe}, so we postpone its actual proof until then.
 
Observe that we only assume $A$-flatness of $L$ here, but not necessarily finite dimensions. That $\Alt$ is a quasi-isomorphism even in case $L$ is infinite dimensional follows from an argument as in the proof of \cite[Thm.~3.2.2]{Lod:CH}, see also \cite[\S{XIII.7}]{CarEil:HA} for the free case.

In the example section \S\ref{examples}, we are going to see that the HKR map not only induces a map of mixed complexes (as already noted in, {\em e.g.}, \cite[Thm.~3.13]{Kow:BVASOCAPB}) but in particular prove in Theorem \ref{transactions} below that it yields an isomorphism of {\em noncommutative differential calculi} or, synonymously, of {\em BV modules}. To this end, we have to deal with the cyclic cohomology resp.\ homology theories of the complexes computing $\Ext$ and  $\Coext$ first, establishing the latter as the {\em cyclic dual} (in the sense of \S\ref{boing}) of the former.

\section{The complex computing $\Ext$ as a cocyclic module}

\subsection{Anti Yetter-Drinfel'd contramodules}
\label{atacvantaggi}

In most cyclic theories, not only the ones including a Hopf structure on the underlying ring or coring, to obtain a para-(co)cyclic object of any kind the possible coefficients typically exhibit more than one algebraic structure, for example, they need to be both modules and comodules or both modules and contramodules. In many cases, these double structures are not immediately recognised as such since one of them might be trivial and therefore invisible as happens in \S\ref{nirvana} for cocommutative bialgebroids and hence for $V\!L$, for example. In any case, to pass from para-(co)cyclic modules to truly cyclic ones, that is, such that
the (co)cyclic operator powers to the identity,
a compatibility condition between these two algebraic structures is required. In the case at hand, we are interested in the following definition from \cite[Def.~4.3]{Kow:WEIABVA}:

\begin{dfn}
\label{chelabertaschen1}
An {\em anti Yetter-Drinfel'd (aYD) contramodule} $M$ over a left Hopf algebroid $(U,A)$ is simultaneously a left $U$-module (with action simply denoted by juxtaposition) and a right $U$-contramodule (with contraaction denoted by $\gamma$) such that both underlying $A$-bimodule structures from \eqref{pergolesi} and \eqref{gaeta} coincide, 
that is,
\begin{equation}
\label{romaedintorni}
a \lact m \ract b = amb, \qquad m \in M, \ a,b \in A,
\end{equation}
and such that action and contraaction are compatible in the sense that 
\begin{equation}
\label{nawas1}
u (\gamma(f)) = \gamma \big(u_{+(2)} f(u_-\sma{-}u_{+(1)}) \big), \qquad \forall u \in U, \ f \in \Hom_\Aopp(U,M).
\end{equation}
An anti  Yetter-Drinfel'd contramodule is called {\em stable} if 
\begin{equation}
\label{stablehalt}
\gamma(\sma{-}m)= m
\end{equation}
for all $m \in M$, where we denote 
$\sma{-}m \colon u \mapsto um$ as a map in $ \Hom_\Aopp(U,M)$.
\end{dfn}

\begin{example}
  \label{shaundasschaf2}
  Again, it is easy to see that for a cocommutative left bialgebroid $(U,A)$ with a left Hopf structure {\em any} left $U$-module $M$ becomes a stable aYD contramodule with respect to the trivial contraaction that evaluates at $1_U$, see Example \ref{trivial}: Eqs.~\eqref{romaedintorni} and \eqref{stablehalt} are immediate and Eq.~\eqref{nawas1} becomes
\begin{equation*}
  \begin{split}
  u (\gamma(f)) =
  u f(1_U) &= \big(u_{(1)+} f(u_{(1)-} u_{(2)})\big)
  \\
  &
=  \big(u_{+(2)} f(u_- 1_U u_{+(1)})\big)
= \gamma \big(u_{+(2)} f(u_-\sma{-}u_{+(1)}) \big),
\end{split}
  \end{equation*}
by cocommutativity along with Eqs.~\eqref{Sch3} and \eqref{Sch4} in Appendix \ref{nebula}. In particular, for 
the universal enveloping algebra $V\!L$ of a Lie-Rinehart algebra $(A,L)$, any left $V\!L$-module is automatically a stable aYD contramodule over $V\!L$.
  \end{example}

A similar definition in the realm of Hopf algebras appeared first in \cite{Brz:HCHWCC}, whereas 
for Hopf algebroids to our knowledge first in \cite{Kow:WEIABVA}.
We refer to 
{\em op.~cit.}, p.~1093, 
for more information about the (not so obvious) well-definedness of Eq.~\eqref{nawas1} and further implications. In particular, one can show that
\begin{equation}
  \label{hatschi}
\gamma(a \lact f\sma{-}) = 
\gamma\big(f(a \blact -)\big),
\end{equation}
where on the left hand side the left $A$-action on $M$ is meant.

The category $ {}_\uhhu \mathbf{aYD}^{\scriptscriptstyle{\rm contra-}\uhhu}$ of right aYD contramodules over a left bialgebroid $U$ in general is not monoidal, considering the fact that in finite dimensions this category is equivalent to that of left modules over the (right) dual $U^*$, see \cite[Lem.~4.6]{Kow:WEIABVA}, which is known not to be monoidal except for some special cases. However, similar to the case of aYD {\em modules}, the category $ {}_\uhhu \mathbf{aYD}^{\scriptscriptstyle{\rm contra-}\uhhu} $ is a module category over $\yd$, the category of Yetter-Drinfel'd (YD) modules over $U$, see \cite[Def.~4.2]{Schau:DADOQGHA}.
More precisely, with the following
we improve Proposition 4.8 in \cite{Kow:WEIABVA} by removing the finiteness condition: 

\begin{prop}
  \label{fliegenervt}
Let $(U,A)$ be a left bialgebroid.
\begin{enumerate}
\item
The operation
  \begin{equation}
    \label{diefliegeistimmernochda}
    \begin{array}{rcl}
      \ucomod \times \contramodu &\to& \contramodu,
      \\
      (N,M) &\mapsto&
      N
            \varogreaterthan
            M :=
      \Hom_\Aopp(N, M)
\end{array}
    \end{equation}
  defines on $\contramodu$ the structure of a module category over the monoidal category $\ucomod$.
\item
  The operation \eqref{diefliegeistimmernochda} restricts to a left action
  $$
\yd \times {}_\uhhu \mathbf{aYD}^{\scriptscriptstyle{\rm contra-}\uhhu} \to {}_\uhhu \mathbf{aYD}^{\scriptscriptstyle{\rm contra-}\uhhu}.
  $$
Hence, $ {}_\uhhu \mathbf{aYD}^{\scriptscriptstyle{\rm contra-}\uhhu}$ is a module category over the monoidal category $\yd$.
\end{enumerate}
\end{prop}


\begin{proof}
  As for the first part, we have to show that for a left $U$-comodule $N$ and $M$ a right $U$-contramodule, $\Hom_\Aopp(N,M)$ can be endowed with a right contraaction as well. Once this contraaction is defined, from the adjunction $\Hom_\Aopp(N' \otimes_A N, M) \simeq \Hom_\Aopp(N', \Hom_\Aopp(N, M))$ for $N, N' \in \ucomod$, one then obtains
  $(N' \otimes_A N) \varogreaterthan M = N' \varogreaterthan (N \varogreaterthan M)$ and hence, $\otimes_A$ being the monoidal product in $\ucomod$, the claim.

  In order to define a right $U$-contraaction on $\Hom_\Aop(N,M)$, let $\gl_N: n \mapsto n_{(-1)} \otimes_A n_{(0)}$ denote the left $U$-coaction on $N$ whereas $\gamma_M$ the $U$-contraaction on $M$, and consider $\Hom_\Aopp(N,M)$ as a right $A$-module by $(ha)(n) := h(an)$ for $a \in A$ and $h \in \Hom_\Aopp(N,M)$. The following then defines a $U$-contraaction on $\Hom_\Aopp(N,M)$:
  \begin{equation}
    \label{streitumasterix}
    \begin{split}
  \gamma: \Hom_\Aopp(U, \Hom_\Aopp(N, M)) 
  &\to \Hom_\Aopp(N, M), \\
  f &\mapsto \big\{ n \mapsto \gamma_M\big(f(n_{(-1)}\sma{-} \otimes_A n_{(0)})\big),
    \end{split}
    \end{equation}
  with the adjunction $\Hom_\Aopp(U, \Hom_\Aopp(N, M)) \simeq \Hom_\Aopp(U_\ract \otimes_A N, M)$ implicitly understood. To show that this indeed defines a contraaction, we will make use of the fact that $\gamma_M$ is already a contraaction, {\em i.e.}, that Eqs.~\eqref{passionant}--\eqref{carrefour2} hold for $\gamma_M$. The right $A$-linearity \eqref{passionant} follows for $\gamma$ by simply observing $\gl_N(na) = n_{(-1)} \bract a \otimes_A n_{(0)}$ along with the right $A$-module structure on $\Hom_\Aopp(N,M)$ as above. Furthermore, for $g \in  \Hom_\Aopp(U \otimes_A U, \Hom_\Aopp(N, M))$ and $n \in N$,
   \begin{equation*}
    \begin{split}
      \dot\gamma\big(\ddot\gamma(g( \cdot \otimes_\ahha \cdot\cdot))\big)(n)
      &=  \dot\gamma_M\big(\ddot\gamma_M(g(n_{(-2)} \sma{\cdot} \otimes_\ahha n_{(-1)}  \sma{\cdot\cdot} \otimes_A n_{(0)}))\big)
      \\
      &= \gamma_M\big(g(n_{(-2)}\sma{-}_{(1)} \otimes_\ahha n_{(-1)}\sma{-}_{(2)} \otimes_A n_{(0)})\big)
      \\
      &= \gamma\big(g(\sma{-}_{(1)} \otimes_\ahha \sma{-}_{(2)})\big)(n),
    \end{split}
    \end{equation*}
which is \eqref{carrefour1} for the map $\gamma$ from \eqref{streitumasterix}. In the same spirit one proves \eqref{carrefour2} and therefore, $\gamma$ indeed constitutes a right $U$-contraaction on  $\Hom_\Aopp(N, M)$.

As for the second part, assume now that $N \in \yd$ and $M \in  {}_\uhhu \mathbf{aYD}^{\scriptscriptstyle{\rm contra-}\uhhu}$, that is, both $N, M$ in particular are left $U$-modules. Then $\Hom_\Aopp(N,M)$ becomes a left $U$-module as well by Eq.~\eqref{pmaction}, and in order to prove that $\Hom_\Aopp(N,M)$ even turns into a stable aYD contramodule over $U$, we need to show that this left action is compatible with the right contraaction in the sense of Eqs.~\eqref{romaedintorni}--\eqref{nawas1}.
Let $h \in \Hom_\Aopp(N,M)$ and $a,b \in A$. That $h \ract b = hb$ follows immediately from \eqref{pmaction} and \eqref{Sch9}. On the other hand,
\begin{equation*}
  \begin{split}
    (ah)(n) &
    \overset{\scriptscriptstyle{\eqref{alleskleber}}}{=}
    \gamma\big(h\gve(\sma{-} \bract a) \big)(n)
    \overset{\scriptscriptstyle{\eqref{streitumasterix}}}{=}
\gamma_M\big(h(\gve(n_{(-1)}\sma{-} \bract a)n_{(0)}\big)
\\
&
    \overset{\scriptscriptstyle{\eqref{tamtamdatam}}}{=}
a \lact \gamma_M\big(h(\gve(n_{(-1)}\sma{-})n_{(0)}\big)
=
a \lact \gamma_M\big(h(n \gve(\sma{-})) \big)
\\
&
=
a \lact \gamma_M\big(h(n) \gve(\sma{-}) \big)
    \overset{\scriptscriptstyle{\eqref{carrefour2}}}{=}
a \lact h(n),
  \end{split}
  \end{equation*}
where in the fourth step we used the properties of a bialgebroid counit, counitality and the fact that the coaction maps into a Takeuchi subspace similar to the coproduct as in
\S\ref{bialgebroids}.
This proves \eqref{romaedintorni}. Moreover, using the fact that $N$ is a YD module and hence the compatibility
\begin{equation}
  \label{yd}
 (u_{(1)}n)_{(-1)} u_{(2)} \otimes_A (u_{(1)}n)_{(0)}= 
  u_{(1)} n_{(-1)} \otimes_A u_{(2)} n_{(0)}
\end{equation}
holds between left $U$-action and left $U$-coaction (see \cite[Def.~4.2]{Schau:DADOQGHA}),
one computes for $f \in \Hom_\Aopp(U, \Hom_\Aopp(N,M))$ that
\begin{small}
\begin{eqnarray*}
(u \pmact \gamma(f))(n)
  &
\!\!\!\!\!\!
\overset{{\scriptscriptstyle{\eqref{pmaction}}}}{=}
&
\!\!\!\!\!\!
u_+(\gamma(f)(u_-n))
 \\
 &
\!\!\!\!\!\!
 \overset{{\scriptscriptstyle{\eqref{streitumasterix}}}}{=}
&
\!\!\!\!\!\!
u_+ \big(\gamma_M\big(f((u_-n)_{(-1)}\sma{-} \otimes_A  (u_-n)_{(0)})\big) \big)
\\
&
\!\!\!\!\!\!
\overset{{\scriptscriptstyle{\eqref{nawas1}}}}{=}
&
\!\!\!\!\!\!
\gamma_M\big(u_{++(2)}f((u_-n)_{(-1)}u_{+-} \sma{-} u_{++(1)} \otimes_A  (u_-n)_{(0)})\big)
\\
&
\!\!\!\!\!\!
\overset{{\scriptscriptstyle{\eqref{Sch5}}}}{=}
&
\!\!\!\!\!\!
\gamma_M\big(u_{+(2)}f((u_{-(1)}n)_{(-1)}u_{-(2)} \sma{-} u_{+(1)} \otimes_A  (u_{-(1)}n)_{(0)})\big)
\\
&
\!\!\!\!\!\!
\overset{{\scriptscriptstyle{\eqref{yd}}}}{=}
&
\!\!\!\!\!\!
\gamma_M\big(u_{+(2)}f(u_{-(1)}n_{(-1)} \sma{-} u_{+(1)} \otimes_A  u_{-(2)}n_{(0)})\big)
\\
&
\!\!\!\!\!\!
\overset{{\scriptscriptstyle{\eqref{Sch4}, \eqref{pmaction}}}}{=}
&
\!\!\!\!\!\!
\gamma_M\big((u_{(2)} \pmact f)(n_{(-1)} \sma{-} u_{(1)} \otimes_A n_{(0)})\big)
\\
&
\!\!\!\!\!\!
\overset{{\scriptscriptstyle{\eqref{streitumasterix}}}}{=}
&
\!\!\!\!\!\!
\gamma\big((u_{(2)} \pmact f)(\sma{-} u_{(1)})\big)(n),
\end{eqnarray*}
\end{small}
which by \eqref{Sch4} again is \eqref{nawas1} for $\Hom_\Aopp(N,M)$ with $U$-action \eqref{pmaction} and $U$-contraaction \eqref{streitumasterix}.
  \end{proof}

\begin{rem}
  The possible stability of the aYD contramodule $\Hom_\Aopp(N,M)$ does {\em not} automatically follow from the possible stability of the aYD contramodule $M$: the stability condition for $\Hom_\Aopp(N,M)$ explicitly reads
  \begin{equation}
    \label{misosuppekyoto}
    \gamma(\sma{-} \pmact h)(n) = \gamma_M\big((n_{(-1)}\sma{-}) \pmact h(n_{(0)})\big) = h(n)
  \end{equation}
  for $h \in \Hom_\Aopp(N,M)$. Even in case of a Hopf algebra over a commutative ring $k$ with involutive antipode $S$, considering $M=k$ as a stable aYD contramodule with trivial action and trivial contraaction (Example \ref{trivial}), the left hand side in \eqref{misosuppekyoto} reads
  $ \gamma(\sma{-} \pmact h)(n) = h(S(n_{(-1)}) n_{(0)})$,
  which in general is different from the right hand side $h(n)$.
  \end{rem}

\subsection{The cocyclic module}
\label{nachher}
In \cite[\S4.2]{Kow:WEIABVA}, for a left Hopf algebroid $(U,A)$ and a left $U$-module right $U$-contramodule $M$, we defined a para-cocyclic $k$-module structure on
\begin{equation}
  \label{mostaccioli}
C^\bullet(U,M) := \Hom_\Aopp(U^{\otimes_\Aopp\bullet}, M),
\end{equation}
where the tensor products are taken with respect to the $A$-bimodule structure $\due U \blact \ract$. Explicitly, in degree $q \in \N$ the structure maps are given by
\begin{small}
\begin{equation}
  \label{nadennwommama}
\begin{array}{rcl}
  (\gd_i f)(u^1, \ldots, u^{q+1}) \!\!\!\!
  &=\!\!\!\!&
  \left\{\!\!\!
\begin{array}{l} 
u^1 f(u^2, \ldots, u^{q+1})
\\ 
 f(u^1, \ldots, u^{i} u^{i+1}, \ldots, u^{q+1})
\\
 f(u^1, \ldots, \gve(u^{q+1}) \blact u^q)
\end{array}\right.  
  \begin{array}{l} \mbox{if} \ i=0, \\ \mbox{if} \
  1 \leq i \leq q, \\ \mbox{if} \ i = q + 1,  \end{array} 
\\
\
\\
(\gs_j f)(u^1, \ldots, u^{q-1}) \!\!\!\! 
&=\!\!\!\!& f(u^1, \ldots, u^{j}, 1, u^{j+1}, \ldots, u^{q-1}) \quad \mbox{for} \  0 \leq j \leq q-1,
\\
\
\\
  (\tau f)(u^1, \ldots, u^q) \!\!\!\!&=\!\!\!\!& \gamma\big(((u^1_{(2)} \cdots u^{q-1}_{(2)} u^q) \pmact  
  f)(-, u^1_{(1)}, \ldots, u^{q-1}_{(1)}) \big),
\end{array}
\end{equation}
\end{small}
the cosimplicial part of which computes the $\Ext$ functor in case $U_\ract$ is flat as an $A$-module, that is, $H(C^\bullet(U,M), \gd) \simeq \Ext^\bullet_U(A,M)$, where as always $\gd := \sum_{i=0}^{n+1} (-1)^i \gd_i$. This para-cocyclic $k$-module becomes cyclic if $M$ is a stable aYD contramodule.

%

In view of \S\ref{atacvantaggi}, we can now fill in more general coefficients in the first entry:

\begin{prop}
  \label{arteepolitica}
  Let $(U, A)$ be a left Hopf algebroid, $N$ a left $U$-module left $U$-comodule, and $M$ a  left $U$-module right $U$-contramodule. Then
$$
  C^\bullet(U \otimes_\Aopp N,M) := \Hom_\Aopp(U^{\otimes_\Aopp\bullet} \otimes_\Aopp N, M)
$$
  can be given the structure of a para-cocyclic $k$-module (the cohomology of which computes $\Ext^\bullet_U(N,M)$ if $U_\ract$ is $A$-flat), which is cyclic if $N$ is a YD module, $M$ an aYD contramodule, and $\Hom_\Aopp(N,M)$ is stable as in \eqref{misosuppekyoto}. In particular, the cyclic coboundary induces an operator
  $$
B: \Ext^\bullet_U(N,M) \to \Ext_U^{\bullet-1}(N,M),
$$
which squares to zero.
  \end{prop}

\begin{proof}
  We simply have to transport the structure maps \eqref{nadennwommama} on $C^\bullet(U, \Hom_\Aopp(N,M))$ for the contramodule $\Hom_\Aopp(N,M)$ to the space $C^\bullet(U \otimes_\Aopp N,M)$ by the correct isomorphism of $k$-modules that produces the correct underlying cosimplicial $k$-module: this isomorphism between $\Hom_\Aopp(U^{\otimes_\Aopp \bullet}, \Hom_\Aopp(N,M))$ and $\Hom_\Aopp(U^{\otimes_\Aopp \bullet} \otimes_\Aopp N, M)$ is not a simple adjunction but rather the adjunction $\Hom_\Aopp(U^{\otimes_\Aopp \bullet}, \Hom_\Aopp(N,M)) \simeq \Hom_\Aopp((U^{\otimes_\Aopp \bullet}) \otimes_A N, M)$, where on the right hand side the $\Aop$-linearity refers to the right $A$-module structure on $N$,  followed by the $k$-module isomorphism
  \begin{eqnarray*}
    \chi:
    \Hom_\Aopp((U^{\otimes_\Aopp q}) \otimes_A N, M) &\to& \Hom_\Aopp(U^{\otimes_\Aopp q} \otimes_\Aopp N, M),
\\
f \mapsto \big\{(u^1, \ldots, u^q , n) &\mapsto&  f((u^1_{(1)}, \ldots, u^q_{(1)}) | u^1_{(2)} \cdots u^q_{(2)}  n)\big\},
\\
\big\{g(u^1_+, \ldots, u^q_+ ,  u^q_- \cdots  u^1_- n) &\mapsfrom& ((u^1, \ldots, u^q) | n) \big\} \mapsfrom g,
    \end{eqnarray*}
in degree $q \in \N$,
  where on the right hand side the $\Aop$-linearity now refers to the right $A$-module structure on the first tensor factor of $U$. Defining then $\gd'_i := \chi \circ \gd_i \circ \chi^{-1}$ and $\gs'_j := \chi \circ \gs_j \circ \chi^{-1}$ by means of the cofaces and codegeneracies from \eqref{nadennwommama}, a quick computation reveals
\begin{equation*}
  \label{nadennwommamanomma}
\begin{array}{rcl}
  (\gd'_i g)(u^1, \ldots, u^{q+1}, n) \!\!\!\!
  &=\!\!\!\!&
  \left\{\!\!\!
\begin{array}{l} 
u^1 g(u^2, \ldots, u^{q+1}, n)
\\ 
 g(u^1, \ldots, u^{i} u^{i+1}, \ldots, u^{q+1}, n)
\\
 g(u^1, \ldots, u^{q}, u^{q+1}n)
\end{array}\right.  
  \begin{array}{l} \mbox{if} \ i=0, \\ \mbox{if} \
  1 \leq i \leq q, \\ \mbox{if} \ i = q + 1,  \end{array} 
  \\[14pt]
(\gs'_j g)(u^1, \ldots, u^{q-1}, n) \!\!\!\! 
&=\!\!\!\!& f(u^1, \ldots, u^{j}, 1, u^{j+1}, \ldots, u^{q-1}, n) \quad \mbox{for} \  0 \leq j \leq q-1,
\end{array}
\end{equation*}
for the cosimplicial $k$-module structure on $C^\bullet(U \otimes_\Aopp N,M)$,
which clearly yields a cochain complex that computes $\Ext_U^\bullet(N,M)$ if $U_\ract$ is $A$-flat. Likewise, by putting $\tau' := \chi \circ \tau \circ \chi^{-1}$ we can promote this cosimplicial module to a para-cocyclic one which is cyclic if $N$ is a YD module, $M$ an aYD module, and $\Hom_\Aopp(N,M)$ stable 
as follows directly from the respective property of the structure maps \eqref{nadennwommama}, Proposition \ref{fliegenervt}, and the fact that $\chi$ is a $k$-module isomorphism. Explicitly, the cocyclic operator is given by
\begin{small}
  \begin{equation*}
    \begin{split}
(\tau' g)(u^1, \ldots, u^q, n) =
      \gamma_M\Big( & u^1_{(2)+} \cdots u^{q-1}_{(2)+} u^q_+ g\big( (n_{(-1)} u^q_- u^{q-1}_{(2)-} \cdots u^{1}_{(2)-}\sma{\cdot})_+, u^1_{(1)+},
      \\
      &
      \ldots,
      u^{q-1}_{(1)+},  u^1_{(1)-} \cdots u^{q-1}_{(1)-} (n_{(-1)} u^q_- u^{q-1}_{(2)-} \cdots u^{1}_{(2)-}\sma{\cdot})_- n_{(0)} \big)
\Big),
    \end{split}
  \end{equation*}
  \end{small}
where $\gamma_M$ denotes the contraaction on $M$ and which, however, is neither nice nor really helpful to stare at but at least reduces to $\tau f$ in \eqref{nadennwommama} again if $N = A$.
  
  The statement about the cyclic coboundary follows by a standard argument involving an SBI sequence, see \cite[\S2.2]{Lod:CH}.
  \end{proof}

\section{The complex computing $\Coext$ as a cyclic module}

In the rest of this article, for the mere sake of simplicity to avoid too messy formul\ae, we restrict ourselves to the case in which $N$ equals the base algebra $A$
itself, with left $U$-action given by $ua := \gve(u \bract a)$ for $u \in U, a \in A$, and left $U$-coaction given by the source map. 

The aim of this section is to compute the cyclic dual in the sense of \S\ref{boing} of the cocyclic module $C^\bullet(U,M)$ from \eqref{mostaccioli}--\eqref{nadennwommama}, where $M$ is a stable aYD contramodule. Merely applying the formula for cyclic duality in \eqref{cyclicdual} does not quite yield the desired result as we are interested in obtaining a cyclic structure on the complex $\Cohom_U(U^{\otimes_A \bullet +1}, M) \simeq D_\bullet(U,M)$
as in \eqref{dee} that computes $\Coext$ (if $M$ is injective as a right $A$-module) by means of the cobar resolution using coproducts, which, as a $k$-module, is quite different from $\Hom_U(U^{\otimes_\Aopp \bullet +1}, M) \simeq C^\bullet(U,M)$, which computes $\Ext$ by means of the bar resolution using products. To circumvent this problem, one uses $k$-linear isomorphisms which transform one complex into the other and which are basically higher order Hopf-Galois maps. From a more abstract point of view,
the cyclic operator arises from a distributive law between two monads, and the isomorphism from the following Lemma maps one monad into the other. Since our main goal in \S\ref{bononiae} is to obtain a chain complex on which the cochain complex $\coc^\bullet(U,A)$ acts in a natural way, it turns out to be more constructive to detect the cyclic structure on 
$$
C_\bullet(U,M) := \Hom_A(U^{\otimes_A \bullet}, M),
$$
where the $A$-linearity refers to the $A$-module structure $\due U \lact {}$ on the first tensor factor, and to connect it to the chain complex $D_\bullet(U,M)$ afterwards.

The subsequent lemma is a straightforward verification relying on Hopf-Galois yoga, that is, on the identities \eqref{Sch1}--\eqref{Tch8} which express the compatibilities of the various structure morphisms for left resp.\ right Hopf algebroids, and in particular on the mixed ones in Eqs.~\eqref{mampf1}--\eqref{mampf3}, which deal with the compatibility in case both left and right Hopf structures are present.

\begin{lemma}
  Let $(U, A)$ be both a left and a right Hopf algebroid and $M$ a left $U$-module. Then for each $n \in \N$ there is a $k$-linear isomorphism
   \begin{small}
     \begin{eqnarray}
       \label{dies}
     \xi: C^n(U,M) &\to& C_n(U,M), \\
\nonumber
    g
    &\mapsto& \big\{(u^1|\ldots|u^n) \mapsto
    u^1_+g(u^1_- u^2_+, \ldots, u^{n-1}_- u^n_+ , u^n_-)
    \big\},
  \end{eqnarray}
   \end{small}
the inverse of which being given by
\begin{small}
  \begin{eqnarray}
    \label{das}
     \xi^{-1}: C_n(U,M) &\to& C^n(U,M), \\
\nonumber
     f &\mapsto& \big\{(u^1, \ldots, u^n) \mapsto u^1_{[+]} \cdots u^n_{[+]}
f\big(u^n_{[-]} \mancino ( \cdots     \mancino (u^2_{[-]} \mancino (u^1_{[-]}|1) |1) \cdots| 1) \big)
\big\},   
     \end{eqnarray}
\end{small}
where $\mancino$ denotes the diagonal $U$-action \eqref{mancino} on $U^{\otimes_A n}$.
\end{lemma}

These isomorphisms allow to obtain the structure maps of a cyclic $k$-module on $C_\bullet(U,M)$ as calculated in the next lemma, which again is achieved by computation only.

\begin{lemma}
  \label{lakritz}
Let $(U, A)$ be both a left and a right Hopf algebroid and $M$ a stable aYD contramodule over $U$. Then the cyclic dual as defined in Eqs.~\eqref{cyclicdual} intertwined by the isomorphism \eqref{dies} obtained from the cocyclic $k$-module $C^\bullet(U,M)$ with structure maps \eqref{nadennwommama} produces in degree $n \in \N$ the following morphisms 
\begin{eqnarray}
  \label{lakritz1}
(d_i f)(u^1|\ldots|u^{n-1}) \!\!\!\!&=\!\!\!\!& \left\{\!\!\!
\begin{array}{l} 
  \gamma\big(\smap f(\smam |u^1 |\ldots |u^{n-1})\big)
\\ 
 f(u^1 |\ldots |\gD u^i |\ldots | u^{n-1})
\\
f(u^1|\ldots|u^{n-1} |1)
\end{array}\right.  
  \begin{array}{l} \mbox{if} \ i=0, \\ \mbox{if} \
  1 \leq i \leq n-1, \\ \mbox{if} \ i = n,  \end{array} 
  \\[3pt]
  \label{lakritz2}
 (s_j f)(u^1 | \ldots | u^{n+1}) \!\!\!\! 
  &=\!\!\!\!& f(u^1 | \ldots | \gve(u^{j+1}) | \ldots | u^{n+1}) \ \quad\qquad\mbox{for} \ \, 0 \leq j \leq n,
  \\[3pt]
  \label{lakritz3}
(t f)(u^1 | \ldots | u^n) \!\!\!\!&=\!\!\!\!& \gamma\Big(((\sma{-} u^1) \mpact
  f)\big(u^2 | \ldots | u^n | 1 \big) \Big)
\end{eqnarray}
on $C_n(U,M)$, where $\mpact$ denotes the left $U$-action on
$\Hom_A(U^{\otimes_A n}, M)$
as in $\eqref{mpaction}$, considering $U^{\otimes_A n}$ as a left $U$-module via the diagonal action \eqref{mancino}.
\end{lemma}

That these structural maps are well-defined and in particular have the correct left $A$-linearity is in case of $d_0$ and $t$ not obvious but follows from \eqref{Tch6} together with \eqref{tamtamdatam}.

\begin{proof}[Proof of Lemma \ref{lakritz}]
  Following the mapping rule in \eqref{cyclicdual}, the claim explicitly reads:
  $$
  d_0 = \xi \circ \gs_{n-1} \tau \circ \xi^{-1}, \quad  d_i = \xi \circ \gs_{i-1} \circ \xi^{-1},
  \quad  s_j = \xi \circ \gd_j \circ \xi^{-1},  \quad  t = \xi \circ \tau^{-1} \circ \xi^{-1}
$$
  for
  $1 \leq i \leq n$ and $0 \leq j \leq n$ with respect to the operators $(\gd_j, \gs_i, \tau)$ from Eqs.~\eqref{nadennwommama}. For the simplicial part, we only show how $d_0$ is computed (as this is already fiddly enough) and leave the rest to the reader. Indeed,
  \begin{small}
\begin{eqnarray*}
 & & (d_0 f)(u^1 | \ldots | u^n) =  (\xi \circ \gs_{n-1} \tau \circ \xi^{-1} f)(u^1 | \ldots | u^n)
  \\
  &
\!\!\!\!\!\!
\overset{{\scriptscriptstyle{\eqref{dies}}}}{=}
&
\!\!\!\!\!\!
u^1_+(\gs_{n-1} \tau \circ \xi^{-1} f)(u^1_-u^2_+, \ldots, u^{n-1}_- u^n_+, u^n_-)
 \\
 &
\!\!\!\!\!\!
 \overset{{\scriptscriptstyle{\eqref{nadennwommama}}}}{=}
&
\!\!\!\!\!\!
 u^1_+(\tau \circ \xi^{-1} f)(u^1_-u^2_+, \ldots, u^{n-1}_- u^n_+, u^n_-, 1)
\\
&
\!\!\!\!\!\!
\overset{{\scriptscriptstyle{\eqref{nadennwommama}}}}{=}
&
\!\!\!\!\!\!
u^1_+\gamma\big(((u^1_{-(2)} u^2_{+(2)} \cdots u^{n-1}_{-(2)} u^n_{+(2)} u^n_{-(2)}) \pmact
\\
&&
(\xi^{-1} f))(\sma{-}, u^1_{-(1)} u^2_{+(1)}, \ldots, u^{n-1}_{-(1)} u^n_{+(1)}, u^n_{-(1)})\big)
\\
&
\!\!\!\!\!\!
\overset{{\scriptscriptstyle{\eqref{Sch5}, \eqref{Sch2}}}}{=}
&
\!\!\!\!\!\!
u^1_+\gamma\big((u^1_{-(2)} \pmact (\xi^{-1} f))(\sma{-}, u^1_{-(1)} u^2_{+}, \ldots, u^{n-1}_{-} u^n_{+}, u^n_{-})\big)
\\
&
\!\!\!\!\!\!
\overset{{\scriptscriptstyle{\eqref{mpaction}}}}{=}
&
\!\!\!\!\!\!
u^1_+\gamma\big(u^1_{-(2)+} (\xi^{-1} f)(u^1_{-(2)-}\sma{-}, u^1_{-(1)} u^2_{+}, \ldots, u^{n-1}_{-} u^n_{+}, u^n_{-})\big)
\\
&
\!\!\!\!\!\!
\overset{{\scriptscriptstyle{\eqref{nawas1}}}}{=}
&
\!\!\!\!\!\!
\gamma\big(u^1_{++(2)}u^1_{-(2)+} (\xi^{-1} f)(u^1_{-(2)-}u^1_{+-}\sma{-}u^1_{++(1)}, u^1_{-(1)} u^2_{+}, \ldots, u^{n-1}_{-} u^n_{+}, u^n_{-})\big)
\\
&
\!\!\!\!\!\!
\overset{{\scriptscriptstyle{\eqref{Sch5}, \eqref{Sch3}}}}{=}
&
\!\!\!\!\!\!
\gamma\big(u^1_{++(2)}u^1_{+-} (\xi^{-1} f)(\sma{-}u^1_{++(1)}, u^1_{-} u^2_{+}, \ldots, u^{n-1}_{-} u^n_{+}, u^n_{-})\big)
\\
&
\!\!\!\!\!\!
\overset{{\scriptscriptstyle{\eqref{Sch2}}}}{=}
&
\!\!\!\!\!\!
\gamma\big((\xi^{-1} f)(\sma{-}u^1_{+}, u^1_{-} u^2_{+}, \ldots, u^{n-1}_{-} u^n_{+}, u^n_{-})\big)
\\
&
\!\!\!\!\!\!
\overset{{\scriptscriptstyle{\eqref{das}}}}{=}
&
\!\!\!\!\!\!
\gamma\Big(\sma{-}_{[+]}u^1_{+[+]} u^1_{-[+]} \cdots u^n_{+[+]} u^n_{-[+]}
\\
&&
\qquad
f\big(u^n_{-[-]} \mancino ((u^n_{+[-]} u^{n-1}_{-[-]}) \mancino ( \cdots \mancino ((u^2_{+[-]} u^{1}_{-[-]}) \mancino (u^1_{+[-]} \sma{-}_{[-]}|1)|1) \cdots |1)\big)\Big)
\\
&
\!\!\!\!\!\!
\overset{{\scriptscriptstyle{\eqref{mampf2}, \eqref{mampf1}, \eqref{Sch7}}}}{=}
&
\!\!\!\!\!\!
\gamma\Big(\sma{-}_{[+]}\gve(u^1_{(1)[+]}) \cdots \gve(u^n_{(1)[+]})
\\
&&
\qquad
f\big(u^n_{(2)} \mancino ((u^n_{(1)[-]} u^{n-1}_{(2)}) \mancino ( \cdots \mancino ((u^2_{(1)[-]} u^{1}_{(2)}) \mancino (u^1_{(1)[-]} \sma{-}_{[-]}|1)|1) \cdots |1)\big)\Big)
\\
&
\!\!\!\!\!\!
\overset{{\scriptscriptstyle{\eqref{Tch4}, \eqref{Tch1}}}}{=}
&
\!\!\!\!\!\!
\gamma\Big(\sma{-}_{[+]}f\big(u^n_{[+]} \mancino ((u^n_{[-]} u^{n-1}_{[+]}) \mancino ( \cdots \mancino ((u^2_{[-]} u^{1}_{[+]}) \mancino (u^1_{[-]} \sma{-}_{[-]}|1)|1) \cdots |1)\big)\Big)
\\
&
\!\!\!\!\!\!
\overset{{\scriptscriptstyle{\eqref{Tch2}}}}{=}
&
\!\!\!\!\!\!
\gamma\big(\sma{-}_{[+]}f(\sma{-}_{[-]} | u^1 | \ldots | u^n )\big).
\end{eqnarray*}
\end{small}
Observe that the aYD condition \eqref{nawas1} was used in line seven and left $A$-linearity of $f$ in the penultimate line. Finally, the computation of $t$ runs along the same lines taking into consideration that the inverse of $\tau$ in \eqref{nadennwommama} is given by
  $$
(\tau^{-1} f)(u^1, \ldots, u^n) = \gamma\big(u^1_+ f(u^2_+, \ldots, u^n_+,  u^n_- \cdots  u^1_- \sma{-}) \big),
$$
see \cite{Kow:WEIABVA}, Eq.~(4.19), where it is denoted by $\tau$ due to the use of an opposite convention. 
  \end{proof}

\begin{rem}
  For later use, we also want to mention the inverse of $t$ if $M$ is a stable aYD contramodule, defined by $t^{-1} := \xi \circ \tau \circ \xi^{-1}$. A direct computation using \eqref{nadennwommama}, \eqref{dies}, and \eqref{das} yields
  \begin{equation}
    \label{kulturkapelle}
(t^{-1} f)(u^1 | \ldots | u^n) = \gamma\Big(\smadotp (u^n \mpact
  f)\big(\smadotm | u^1 | \ldots | u^{n-1}\big) \Big)
    \end{equation}
for $f \in C_n(U,M)$.
\end{rem}

\begin{theorem}
  \label{salveregina}
If $(U,A)$ is both a left and a right Hopf algebroid and $M$ a stable aYD contramodule over $U$, then the cyclic dual $\big(C_\bullet(U,M), d_\bullet, s_\bullet, t
  \big)$
defines a cyclic $k$-module the simplicial part of which induces a chain complex isomorphic to the chain complex $\big(D_\bullet(U,M), b\big)$ as in \eqref{landliebekirsch}--\eqref{dee} that computes $\Coext_\bullet^U(A,M)$ if $M$ is $A$-injective. 
  \end{theorem}
\begin{proof}
  The first statement is a tautological consequence of how cyclic duals are constructed (the cyclic dual of a cocyclic module being a cyclic module) along with the fact that the maps \eqref{dies} and \eqref{das} are isomorphisms of $k$-modules.

  As for the second part, we want to show that  $(C_\bullet(U,M), d) \simeq (D_\bullet(U,M), b)$ as chain complexes, where $d = \sum_{i=0}^n (-1)^n d_i$ in degree $n$ for the faces in Eqs.~\eqref{lakritz1}. To this end, 
  consider first the following $k$-linear isomorphism
     \begin{small}
     \begin{eqnarray}
\nonumber
       \zeta: \Hom_A(U^{\otimes_A n}, M) &\to&  \Hom_\Aopp(U^{\otimes_A n}, M), \\
\nonumber
    f
    \mapsto \big\{(u^1|\ldots|u^n) &\mapsto&
    (u^n \mpact f)(1|u^1| \ldots| u^{n-1})
    \big\},
    \\
  \label{tat}
    \big\{ (u^1 \pmact g)(u^2| \ldots| u^n|1) &\mapsfrom& (u^1|\ldots|u^n) \} \mapsfrom g.
     \end{eqnarray}
\end{small}
%
%
     Coupling then this isomorphism with the cyclic operator $t$, which is an isomorphism as well if $M$ is a stable aYD module --- with inverse quoted in Eq.~\eqref{kulturkapelle} --- does the job; that is, defining $\eta := \zeta \circ t$, we obtain an isomorphism $C_\bullet(U,M) \simeq D_\bullet(U,M)$ with the property that
$
\eta \circ d_i = b_i \circ \eta 
$
for all faces, that is, for all $0 \leq i \leq n$ in degree $n$. We only show this for $i =0$ which is the most intricate case, and leave the rest to the reader. On the other hand, it turns out to be more convenient working with the inverse, and we therefore compute $\eta^{-1}$ first: for $g \in \Hom_\Aopp(U^{\otimes_A n}, M)$, we have
\begin{small}
\begin{eqnarray*}
  & &
\!\!\!\!\!\!\!\!\!\!\!\!\!\!\!\!\!\!\!\!\!\!\!\!
  (\eta^{-1} g)(u^1|\ldots|u^n) =  (t^{-1} \zeta^{-1} g)(u^1|\ldots|u^n)
  \\
  &
\!\!\!\!\!\!
\overset{{\scriptscriptstyle{\eqref{kulturkapelle}}}}{=}
&
\!\!\!\!\!\!
\gamma\Big(\smadotp (u^n \mpact (\zeta^{-1}g))(\smadotm|u^1|\ldots|u^{n-1}) \Big)
 \\
 &
\!\!\!\!\!\!
 \overset{{\scriptscriptstyle{\eqref{Tch5}}}}{=}
&
\!\!\!\!\!\!
\gamma\Big(\smadotp u^n_{[+]} \big(\zeta^{-1}g\big)\big(u^n_{[+][-]}\smadotm| u^n_{[-]} \mancino(u^1|\ldots|u^{n-1})\big) \Big)
\\
&
\!\!\!\!\!\!
\overset{{\scriptscriptstyle{\eqref{tat}}}}{=}
&
\!\!\!\!\!\!
\gamma\Big(\smadotp u^n_{[+][+]} \big((u^n_{[+][-]}\smadotm) \pmact g\big)\big(u^n_{[-]} \mancino(u^1|\ldots|u^{n-1})|1\big) \Big)
\\
&
\!\!\!\!\!\!
\overset{{\scriptscriptstyle{\eqref{pmaction}}}}{=}
&
\!\!\!\!\!\!
\gamma\Big(\smadotp u^n_{[+][+]}u^n_{[+][-]+}(\smadotm \pmact g\big)\big((u^n_{[+][-]-(1)}u^n_{[-]}) \mancino(u^1|\ldots|u^{n-1})| u^n_{[+][-]-(2)}\big) \Big)
\\
&
\!\!\!\!\!\!
\overset{{\scriptscriptstyle{\eqref{mampf3}, \eqref{Tch7}}}}{=}
&
\!\!\!\!\!\!
\gamma\Big(\smadotp (\smadotm \pmact g\big)\big((u^n_{[+](1)}u^n_{[-]}) \mancino(u^1|\ldots|u^{n-1})| u^n_{[+](2)}\big) \Big)
\\
&
\!\!\!\!\!\!
\overset{{\scriptscriptstyle{\eqref{Tch2}, \eqref{mpaction}}}}{=}
&
\!\!\!\!\!\!
\gamma\Big(\sma{\cdot}_{{\scriptscriptstyle [+]}} \sma{\cdot}_{\scriptscriptstyle [-]+} g\big(\sma{\cdot}_{\scriptscriptstyle [-]-} \mancino(u^1|\ldots| u^n) \big) \Big)
\\
&
\!\!\!\!\!\!
\overset{{\scriptscriptstyle{\eqref{mampf3}, \eqref{Tch7}}}}{=}
&
\!\!\!\!\!\!
\gamma\Big(g\big(\sma{\cdot} \mancino(u^1|\ldots| u^n) \big) \Big),
\end{eqnarray*}
\end{small}
using right $A$-linearity of $g$ in line six and eight. With this, we compute on one side
\begin{small}
\begin{equation*}
  (\eta^{-1} b_0 g)(u^1|\ldots|u^{n-1})
  = \gamma\Big((b_0g)\big(\sma{\cdot} \mancino(u^1|\ldots| u^{n-1}) \big) \Big)
= \gamma\Big(g\big(1|\sma{\cdot} \mancino(u^1|\ldots| u^{n-1}) \big) \Big),
\end{equation*}
\end{small}
using $b_0$ from \eqref{landliebekirsch} for $N = A$, and on the other side, by means of $d_0$ from \eqref{lakritz1},
\begin{small}
\begin{eqnarray*}
  & &
\!\!\!\!\!\!\!\!\!\!\!\!
  (d_0 \eta^{-1} g)(u^1|\ldots|u^{n-1})
  \\
  &
\!\!\!\!\!\!
\overset{{\scriptscriptstyle{\eqref{lakritz1}}}}{=}
&
\!\!\!\!\!\!
\gamma\Big(\smadotp (\eta^{-1} g)(\smadotm|u^1|\ldots|u^{n-1}) \Big)
 \\
 &
\!\!\!\!\!\!
 \overset{{\scriptscriptstyle{}}}{=}
&
\!\!\!\!\!\!
\dot\gamma\Big(\smadotp \ddot\gamma\big(g(\sma{\cdot\cdot} \mancino (\smadotm|u^1|\ldots|u^{n-1}))\big) \Big)
\\
&
\!\!\!\!\!\!
\overset{{\scriptscriptstyle{\eqref{nawas1}}}}{=}
&
\!\!\!\!\!\!
\dot\gamma\Big(\ddot\gamma\big(\sma{\cdot}_{\scriptscriptstyle [+]+(2)} g((\sma{\cdot}_{\scriptscriptstyle [+]-} \sma{\cdot\cdot}\sma{\cdot}_{\scriptscriptstyle [+]+(1)} ) \mancino (\smadotm|u^1|\ldots|u^{n-1}))\big) \Big)
\\
&
\!\!\!\!\!\!
\overset{{\scriptscriptstyle{\eqref{carrefour1}, \eqref{mampf1}}}}{=}
&
\!\!\!\!\!\!
\gamma\Big(\sma{\cdot}_{\scriptscriptstyle (1)+[+](2)} g\big((\sma{\cdot}_{\scriptscriptstyle (1)-} \sma{\cdot}_{\scriptscriptstyle (2)}\sma{\cdot}_{\scriptscriptstyle (1)+[+](1)} ) \mancino (\sma{\cdot}_{\scriptscriptstyle (1)+[-]}|u^1|\ldots|u^{n-1})\big) \Big)
\\
&
\!\!\!\!\!\!
\overset{{\scriptscriptstyle{\eqref{Sch3}}}}{=}
&
\!\!\!\!\!\!
\gamma\Big(\sma{\cdot}_{\scriptscriptstyle [+](2)} g\big(\sma{\cdot}_{\scriptscriptstyle [+](1)}
\mancino (\sma{\cdot}_{\scriptscriptstyle [-]}|u^1|\ldots|u^{n-1})\big) \Big)
\\
&
\!\!\!\!\!\!
\overset{{\scriptscriptstyle{\eqref{Tch2}}}}{=}
&
\!\!\!\!\!\!
\gamma\Big(\sma{\cdot}_{\scriptscriptstyle (2)} g\big(1|\sma{\cdot}_{\scriptscriptstyle (1)} \mancino (u^1|\ldots|u^{n-1})\big) \Big)
\\
&
\!\!\!\!\!\!
\overset{{\scriptscriptstyle{\eqref{carrefour1}}}}{=}
&
\!\!\!\!\!\!
\dot\gamma\Big(\ddot\gamma\big(\sma{\cdot\cdot} g(1|\sma{\cdot} \mancino (u^1|\ldots|u^{n-1}))\big) \Big)
\\
&
\!\!\!\!\!\!
\overset{{\scriptscriptstyle{\eqref{stablehalt}}}}{=}
&
\!\!\!\!\!\!
\gamma\Big(g\big(1|\sma{\cdot} \mancino (u^1|\ldots|u^{n-1})\big) \Big)
\\
&
\!\!\!\!\!\!
\overset{{\scriptscriptstyle{}}}{=}
&
\!\!\!\!\!\!
(\eta^{-1} b_0 g)(u^1|\ldots|u^{n-1}),
\end{eqnarray*}
\end{small}
where we used the aYD property in line four and stability in line nine.
Verifying analogous identities for all faces, we altogether obtain $\eta \circ d = b \circ \eta$ and
hence, $\eta: (C_\bullet(U,M), d) \stackrel{\simeq}{\longrightarrow} (D_\bullet(U,M), b)$ gives the desired isomorphism of chain complexes.
\end{proof}

\section{\for{toc}{Chain complexes as cyclic opposite modules over cochain complexes}\except{toc}{The complex computing $\Coext$ as a cyclic opposite module over the complex computing $\Cotor$}}


\label{bononiae}

In this section, we advance to the core of the article by defining the structure of a cyclic unital opposite module on the chain complex computing $\Coext$ over the operad with multiplication given by the cochain complex computing $\Cotor$ in a sense we briefly recall in \S\ref{dauerregen}. This, by means of  Theorem \ref{terzamissione}, induces a noncommutative calculus (or BV module) up to homotopy in the sense of \S\ref{esregnetdurchdieDecke}, as will be discussed in the next section.

Since we would like to see from now on the cochain resp.\ chain spaces just mentioned from a more operadic point of view, we change the notation and set for $n, p \in \N$
\begin{equation}
  \label{maisondumonde}
\cO(p) := \coc^p(U,A), \qquad \cM(n) :=  C_n(U, M), 
\end{equation}
where $(U,A)$ for the time being is only a left bialgebroid and $M$ a right $U$-contramodule.
The operadic structure of $\cO$ was explicitly described in \cite[Eqs.~(3.3)--(3.5)]{Kow:BVASOCAPB} for general coefficients (more precisely, with coefficients being (braided) commutative monoids in the braided
category of Yetter-Drinfel'd modules, see there). Here, we will only deal with the case of coefficients in the base algebra $A$ but a more general approach would also be possible without too much additional effort.

The operadic structure on $\cO$ is defined by the partial composition maps
$
\circ_i : \cO(p) \otimes  \cO(q) \rightarrow 
         \cO(p+q-1) 
$
         given as
\begin{equation}
\label{maxdudler2}
\begin{split}
(u^1| \ldots| u^p) \circ_i (v^1| \ldots| v^q) := (u^1 | \ldots| u^{i-1}| u^{i} \mancino (v^1| \ldots |v^q) | u^{i+1} | \ldots | u^p)
\end{split}
\end{equation}
for all $1 \leq i \leq p$, where $\mancino$ as always denotes the diagonal action \eqref{mancino} 
given by the $(q-1)$-fold iterated coproduct $\gD^{q-1}$ on elements of degree $q$. For $q =1$, set $\gD^0 = \id_U$ and $\mancino$ becomes the multiplication in $U$, whereas
for $q= 0$, that is, an element in $\cO(0) = A$, set $\gD^{-1} = \gve$, the counit of $U$.
In particular, $\cO$ is an operad with multiplication (see \S\ref{pamukkale1}), the multiplication, the identity, and the unit being given by
\begin{equation}
\label{soschlapp}
  (\mu, \mathbb{1}, e) := \big((1_U | 1_U), 1_U, 1_A\big). 
\end{equation}

In view of what we are aiming at let us already mention at this point that this structure on the cochains by a well-known theorem directly implies that the cohomology $\Cotor^\bullet_U(A,A)$ is a Gerstenhaber algebra, see again \S\ref{pamukkale1}.

As far as the opposite $\cO$-module structure on $\cM$ is concerned, define for all $i = 1, \ldots, n-p+1$ and $0 \leq p \leq n$,
the operation
\begin{equation}
  \label{viatoledo1}
  \begin{array}{rcl}
    \bullet_i: \cO(p) \otimes \cM(n) &\to& \cM(n-p+1), \\
    w \otimes f &\mapsto&
  f\big(\sma{-}^1 | \ldots | \sma{-}^{i-1} | 
\sma{-}^i \mancino w |
\sma{-}^{i+1} | \ldots | \sma{-}^{n}\big),
\end{array}
  \end{equation}
declared to be zero if $p > n$.
%
Explicitly, this has to be read as follows: let $w = (u^1 | \ldots | u^p) \in U^{\otimes_A p}$ and $f \in \Hom_A(U^{\otimes_A n},M)$. Then 
\begin{small}
$$
  \big((u^1 | \ldots | u^p) \bullet_i f\big)(v^1 | \ldots | v^{n-p+1})
  := f\big(v^1 | \ldots | v^{i-1} | v^i_{(1)} u^1 | \ldots |v^i_{(p)} u^p |
v^{i+1} | \ldots | v^{n-p+1}\big)
$$
\end{small}
for all $i = 1, \ldots, n-p+1$ and $(v^1 | \ldots | v^{n-p+1}) \in U^{\otimes_A n-p+1}$, and
where $(v^i_{(1)}| \ldots |v^i_{(p)})$ denotes the $(p-1)$-fold iterated coproduct $\gD^{p-1}(v^i)$.
Again, for an element in $\cO(0) = A$
acting on $f$, one sets $\gD^{-1} = \gve$ as above, and hence
\begin{equation}
  \label{viatoledo2}
\big(a \bullet_i f\big)(v^1 | \ldots | v^{n+1}) := f\big(v^1 | \ldots | v^{i-1} | \gve(v^i \bract a) |
v^{i+1} | \ldots | v^{n+1}\big)
\end{equation}
for $a \in A$ and $i = 1, \ldots, n+1$.

 \begin{lemma}
The operations \eqref{viatoledo1} induce on $\cM$ the structure of a unital opposite $\cO$-module. 
   \end{lemma}

 \begin{proof}
   By definition, we have to check the identities \eqref{TchlesischeStr} for the operations \eqref{viatoledo1} and the operadic structure on $\cO$ spelled out in \eqref{maxdudler2} above. 
   This is an obvious verification that hinges essentially on coassociativity along with the compatibility between product and coproduct, which is why we omit it.
   \end{proof}

 A possible Hopf structure on a left bialgebroid $(U,A)$ comes into play when one wants to promote this opposite module into a cyclic one, in particular when adding the extra operation $\bullet_0$. To this end, let $U$ be both a left and right Hopf algebroid, and set
\begin{equation}
   \label{extraop}
(u^1 | \ldots | u^p) \bullet_0 f :=
\dot\gamma\big(\smadotp (u^1 \mpact f)(\smadotm \mancino (u^2 | \ldots | u^p) | \sma{-}^1 | \ldots | \sma{-}^{n-p+1})\big),
   \end{equation}
for $f \!\in\! \cM(n)$, declared to be zero this time if $p > n + 1$.
That this is a well-defined expression, indeed, is possibly not obvious at first sight, but follows from $A$-linearity of the coproduct as well as Eqs.~\eqref{hatschi}, \eqref{Tch1}, \eqref{Tch6}, and \eqref{Tch9}.
Explicitly on an element 
$(v^1 | \ldots | v^{n-p+1}) \in  U^{\otimes_A n-p+1}$ and all short-hand notations written out, this reads as
\begin{small}
\begin{equation*}
  \begin{split}
    \big( & (u^1 | \ldots | u^p) \bullet_0 f\big)(v^1 | \ldots | v^{n-p+1}) \\
 &=
\gamma\big({\sma{\cdot}_{[+]}} u^1_{[+]} f( u^1_{[-](1)} {\sma{\cdot}_{[-](1)}} u^2 | \ldots | u^1_{[-](p-1)} {\sma{\cdot}_{[-](p-1)}} u^p | u^1_{[-](p)} v^1 | \ldots | u^1_{[-](n)} v^{n-p+1})\big).
\end{split}
  \end{equation*}
\end{small}
Then, along with the cyclic operator 
$$
(t f)(v^1| \ldots | v^n) = \gamma\Big(((\sma{\cdot} v^1) \mpact f)\big(v^2| \ldots| v^n| 1 \big) \Big)
$$
from \eqref{lakritz3}, we can turn the opposite $\cO$-module $\cM$ into a cyclic one:

\begin{theorem}
  \label{plebiscito}
If $M$ is a stable aYD contramodule over a left bialgebroid $(U,A)$ which is both left and right Hopf, the extra operation \eqref{extraop} turns $(\cM, t)$ into a cyclic unital opposite module over the operad with multiplication $(\cO, \mu, e)$, the underlying cyclic $k$-module structure of which coincides with the one of the cyclic dual of $C^\bullet(U,M)$ given in Eqs.~\eqref{lakritz1}--\eqref{lakritz3}.
 \end{theorem}

\begin{proof}
We first prove the second statement regarding the cyclic $k$-module structure. As the cyclic operator coincides by construction, we only need to check that the simplicial structure defined in Eqs.~\eqref{lakritz1}--\eqref{lakritz2} coincides with the one originating from being a cyclic opposite $\cO$-module by the general construction in Eqs.~\eqref{colleoppio}: with $\mu = (1_U | 1_U)$, one immediately sees from \eqref{extraop} and \eqref{viatoledo1} that
 \begin{eqnarray}
\label{schlappschlappschlapp}
   (\mu \bullet_0 f)(v^1 |\ldots |v^{n-1}) 
&=& \gamma\big(\smadotp f(\smadotm |v^1 |\ldots |v^{n-1})\big) \\
\nonumber
(\mu \bullet_i f)(v^1 |\ldots |v^{n-1}) &=& f(v^1 |\ldots |\gD v^i |\ldots |v^{n-1}), \qquad 
    \end{eqnarray}
for $i = 1, \ldots, n$,
which are the first two lines in Eqs.~\eqref{lakritz1}; as for the last face, compute
\begin{small}
\begin{eqnarray*}
 & &(\mu \bullet_0 tf)(v^1 |\ldots |v^{n-1})
  \\
  &
\!\!\!\!\!\!
\overset{{\scriptscriptstyle{\eqref{schlappschlappschlapp}}}}{=}
&
\!\!\!\!\!\!
 \gamma\big(\smadotp tf(\smadotm |v^1 |\ldots |v^{n-1})\big)
 \\
 &
\!\!\!\!\!\!
 \overset{{\scriptscriptstyle{\eqref{lakritz3}}}}{=}
&
\!\!\!\!\!\!
\dot\gamma\Big(\smadotp \ddot\gamma\big({\scriptstyle{(\cdot\cdot)_{[+]}}} {\scriptstyle{(\cdot)_{[-][+]}}} f(  ({\scriptstyle{(\cdot)_{[-][-]}}} {\scriptstyle{(\cdot\cdot)_{[-]}}}) \mancino (v^1 |\ldots |v^{n-1}|1))\big)\Big)
\\
&
\!\!\!\!\!\!
\overset{{\scriptscriptstyle{\eqref{nawas1}, \eqref{Tch6}}}}{=}
&
\!\!\!\!\!\!
\dot\gamma\Big( \ddot\gamma\big({\scriptstyle{(\cdot)_{[+]+(2)}}}{\scriptstyle{(\cdot)_{[+]-[+]}}} {\scriptstyle{(\cdot\cdot)_{[+]}}} {\scriptstyle{(\cdot)_{[+]+(1)[+]}}}{\scriptstyle{(\cdot)_{[-][+]}}}
\\
&& \qquad \qquad 
f(  ({\scriptstyle{(\cdot)_{[-][-]}}} {\scriptstyle{(\cdot)_{[+]+(1)[-]}}}{\scriptstyle{(\cdot\cdot)_{[-]}}} {\scriptstyle{(\cdot)_{[+]-[-]}}}) \mancino (v^1 |\ldots |v^{n-1}|1))\big)\Big)
\\
&
\!\!\!\!\!\!
\overset{{\scriptscriptstyle{\eqref{mampf2}, \eqref{Sch2}}}}{=}
&
\!\!\!\!\!\!
\dot\gamma\Big( \ddot\gamma\big({\scriptstyle{(\cdot\cdot)_{[+]}}} {\scriptstyle{(\cdot)_{[+](1)[+]}}}{\scriptstyle{(\cdot)_{[-][+]}}}
f(({\scriptstyle{(\cdot)_{[-][-]}}} {\scriptstyle{(\cdot)_{[+](1)[-]}}}{\scriptstyle{(\cdot\cdot)_{[-]}}} {\scriptstyle{(\cdot)_{[+](2)}}}) \mancino (v^1 |\ldots |v^{n-1}|1))\big)\Big)
\\
&
\!\!\!\!\!\!
\overset{{\scriptscriptstyle{\eqref{Tch6}, \eqref{Tch2}}}}{=}
&
\!\!\!\!\!\!
\dot\gamma\Big( \ddot\gamma\big(({\scriptstyle{(\cdot\cdot)_{[+]}}} f(({\scriptstyle{(\cdot\cdot)_{[-]}}} {\scriptstyle{(\cdot)}}) \mancino (v^1 |\ldots |v^{n-1}|1))\big)\Big)
\\
&
\!\!\!\!\!\!
\overset{{\scriptscriptstyle{\eqref{carrefour1}}}}{=}
&
\!\!\!\!\!\!
\gamma\big({\scriptstyle{(\cdot)_{(2)[+]}}} f(({\scriptstyle{(\cdot)_{(2)[-]}}} {\scriptstyle{(\cdot)_{(1)}}}) \mancino (v^1 |\ldots |v^{n-1}|1))\big)
\\
&
\!\!\!\!\!\!
\overset{{\scriptscriptstyle{\eqref{Tch3}}}}{=}
&
\!\!\!\!\!\!
\gamma\big({\scriptstyle{(\cdot)}} f(v^1 |\ldots |v^{n-1}|1)\big)
\\
&
\!\!\!\!\!\!
\overset{{\scriptscriptstyle{\eqref{stablehalt}}}}{=}
&
\!\!\!\!\!\!
f(v^1 |\ldots |v^{n-1}|1),
\end{eqnarray*}
\end{small}
which is the last line in \eqref{lakritz1}. For the degeneracies we obtain for all $j = 0, \ldots, n$ by simply staring at \eqref{viatoledo2} along with \eqref{soschlapp}
$$
(e \bullet_{j+1} f)(v^1 |\ldots |v^{n+1}) = f\big(v^1 | \ldots | \gve(v^{j+1}) | \ldots | v^{n+1}\big),
$$
which is \eqref{lakritz2}.

To conclude the proof, we have to check Eq.~\eqref{lagrandebellezza1} 
  in this situation, that is $t(w \bullet_i f) = w \bullet_{i+1} f$
  for $0 \leq i \leq n-p$ and $w \in \cO(p)$, $f \in \cM(n)$ but we are going to do this only for $i = 0$ as this is the most difficult case; the verification for $1 \leq i \leq n-p$ will be left to the reader.
  Indeed, for $w = (u^1|\ldots|u^p) \in \cO(p)$, we have
  \begin{small}
\begin{eqnarray*}
 & &(t(w \bullet_0 f))(v^1 |\ldots |v^{n-p+1})
  \\
  &
\!\!\!\!\!\!
\overset{{\scriptscriptstyle{\eqref{lakritz3}}}}{=}
&
\!\!\!\!\!\!
\gamma\Big(\big((\sma{\cdot} v^1) \mpact (w \bullet_0 f)\big)\big(v^2| \ldots| v^{n-p+1}| 1 \big) \Big)
 \\
 &
\!\!\!\!\!\!
 \overset{{\scriptscriptstyle{\eqref{extraop}}}}{=}
&
\!\!\!\!\!\!
\dot\gamma\Big((\sma{\cdot} v^1)_{[+]} \ddot\gamma\big(\sma{\cdot\cdot}_{[+]} (u^1 \mpact f)\big(\sma{\cdot\cdot}_{[-]} \mancino (u^2 | \ldots | u^p)| (\sma{\cdot} v^1)_{[-]}  \mancino (v^2| \ldots| v^{n-p+1}| 1) \big)\big) \Big)
\\
&
\!\!\!\!\!\!
\overset{{\scriptscriptstyle{\eqref{nawas1}, \eqref{Tch6}}}}{=}
&
\!\!\!\!\!\!
\dot\gamma\Big(\ddot\gamma\big((\sma{\cdot} v^1)_{[+]+(2)} (\sma{\cdot} v^1)_{[+]-[+]}  \sma{\cdot\cdot}_{[+]}
(\sma{\cdot} v^1)_{[+]+(1)[+]}
(u^1 \mpact f)
\\
&&
\big((\sma{\cdot} v^1)_{[+]+(1)[-]} \sma{\cdot\cdot}_{[-]} (\sma{\cdot} v^1)_{[+]-[-]} \mancino (u^2 | \ldots | u^p)| (\sma{\cdot} v^1)_{[-]}  \mancino (v^2| \ldots| v^{n-p+1}| 1) \big)\big) \Big)
\\
&
\!\!\!\!\!\!
\overset{{\scriptscriptstyle{\eqref{mampf2}, \eqref{Sch2}}}}{=}
&
\!\!\!\!\!\!
\dot\gamma\Big(\ddot\gamma\big(\sma{\cdot\cdot}_{[+]}
(\sma{\cdot} v^1)_{[+](1)[+]}
(u^1 \mpact f)
\\
&&
\big((\sma{\cdot} v^1)_{[+](1)[-]} \sma{\cdot\cdot}_{[-]} (\sma{\cdot} v^1)_{[+](2)} \mancino (u^2 | \ldots | u^p)| (\sma{\cdot} v^1)_{[-]}  \mancino (v^2| \ldots| v^{n-p+1}| 1) \big)\big) \Big)
\\
&
\!\!\!\!\!\!
\overset{{\scriptscriptstyle{\eqref{carrefour1}, \eqref{Tch6}}}}{=}
&
\!\!\!\!\!\!
\gamma\big(\sma{\cdot}_{(2)[+]}
\sma{\cdot}_{(1)[+](1)[+]} v^1_{[+](1)[+]}
(u^1 \mpact f)
\\
&& \!\!\!
\big(v^1_{[+](1)[-]} \sma{\cdot}_{(1)[+](1)[-]} \sma{\cdot}_{(2)[-]} \sma{\cdot}_{(1)[+](2)} v^1_{[+](2)} \mancino (u^2 | \ldots | u^p)| (\sma{\cdot} v^1)_{[-]}  \mancino (v^2| \ldots| v^{n-p+1}| 1) \big)\big)
\\
&
\!\!\!\!\!\!
\overset{{\scriptscriptstyle{\eqref{Tch4}, \eqref{Tch3}}}}{=}
&
\!\!\!\!\!\!
\gamma\big(\sma{\cdot}_{[+](2)}
\sma{\cdot}_{[+](1)[+]} v^1_{[+](1)[+]}
(u^1 \mpact f)
\\
&& \!\!\!
\big(v^1_{[+](1)[-]} \sma{\cdot}_{[+](1)[-]} v^1_{[+](2)} \mancino (u^2 | \ldots | u^p)| ( v^1_{[-]} \sma{\cdot}_{(1)[-]}) \mancino (v^2| \ldots| v^{n-p+1}| 1) \big)\big)
\\
&
\!\!\!\!\!\!
\overset{{\scriptscriptstyle{\eqref{Tch4}}}}{=}
&
\!\!\!\!\!\!
\gamma\big(\sma{\cdot}_{(2)}
\sma{\cdot}_{(1)[+][+]} v^1_{[+](1)[+]}
(u^1 \mpact f)
\\
&& \!\!\!
\big(v^1_{[+](1)[-]} \sma{\cdot}_{(1)[+][-]} v^1_{[+](2)} \mancino (u^2 | \ldots | u^p)| ( v^1_{[-]} \sma{\cdot}_{(1)[-]}) \mancino (v^2| \ldots| v^{n-p+1}| 1) \big)\big)
\\
&
\!\!\!\!\!\!
\overset{{\scriptscriptstyle{\eqref{carrefour1}}}}{=}
&
\!\!\!\!\!\!
\dot\gamma\Big(\ddot\gamma\big(\sma{\cdot\cdot}
\sma{\cdot}_{[+][+]} v^1_{[+](1)[+]}
(u^1 \mpact f)
\\
&& \!\!\!
\big(v^1_{[+](1)[-]} \sma{\cdot}_{[+][-]} v^1_{[+](2)} \mancino (u^2 | \ldots | u^p)| ( v^1_{[-]} \sma{\cdot}_{[-]}) \mancino (v^2| \ldots| v^{n-p+1}| 1) \big)\big)\Big)
\\
&
\!\!\!\!\!\!
\overset{{\scriptscriptstyle{\eqref{stablehalt}}}}{=}
&
\!\!\!\!\!\!
\gamma\big(\sma{\cdot}_{[+][+]} v^1_{[+](1)[+]}
(u^1 \mpact f)
\\
&& \!\!\!
\big(v^1_{[+](1)[-]} \sma{\cdot}_{[+][-]} v^1_{[+](2)} \mancino (u^2 | \ldots | u^p)| ( v^1_{[-]} \sma{\cdot}_{[-]}) \mancino (v^2| \ldots| v^{n-p+1}| 1) \big)\big)
\\
&
\!\!\!\!\!\!
\overset{{\scriptscriptstyle{\eqref{Tch4}, \eqref{Tch5}}}}{=}
&
\!\!\!\!\!\!
\gamma\big((\sma{\cdot} v^1)_{[+]}
(u^1 \mpact f)
\\
&& \!\!\!
\big(((\sma{\cdot} v^1_{(1)})_{[-](1)} v^1_{(2)}) \mancino (u^2 | \ldots | u^p)| (\sma{\cdot} v^1_{(1)})_{[-](2)} \mancino (v^2| \ldots| v^{n-p+1}| 1) \big)\big)
\end{eqnarray*}
\begin{eqnarray*}
&
\!\!\!\!\!\!
\overset{{\scriptscriptstyle{}}}{=}
&
\!\!\!\!\!\!
\gamma\Big(((\sma{\cdot} v^1_{(1)} u^1) \mpact f)
\big(v^1_{(2)} \mancino (u^2 | \ldots | u^p)| v^2| \ldots| v^{n-p+1}| 1\big)\Big)
\\
&
\!\!\!\!\!\!
\overset{{\scriptscriptstyle{\eqref{lakritz3}}}}{=}
&
\!\!\!\!\!\!
tf\big(v^1_{(1)} u^1 | v^1_{(2)} \mancino (u^2 | \ldots | u^p)| v^2| \ldots| v^{n-p+1}\big)
\\
&
\!\!\!\!\!\!
\overset{{\scriptscriptstyle{\eqref{viatoledo1}}}}{=}
&
\!\!\!\!\!\!
(w \bullet_1 tf)(v^1 | \ldots| v^{n+p-1}),
\end{eqnarray*}
  \end{small}
which finishes the proof.
 \end{proof}

\begin{rem}
  As already mentioned, the restriction to trivial coefficients (that is, the base algebra $A$) in the operadic structure of $\coc^p(U,A)$ is not necessary and has only been made to avoid too cumbersome formul{\ae} that might obscure the general idea. Replacing $A$ by a (braided) commutative monoid in the braided
category of Yetter-Drinfel'd modules would also work, see \cite[Thm.~1.3]{Kow:BVASOCAPB}.
   \end{rem}


\section{The noncommutative calculus structure on $\Coext$ over $\Cotor$}

The main purpose of the structure of a cyclic opposite operad module analysed in detail in the previous section stems from the fact that they are directly related to what is called {\em (homotopy) noncommutative differential calculi} or {\em (homotopy) BV modules}. Let us briefly review them first and then see how they are connected to the results in the previous section.

\subsection{Noncommutative differential calculi {\cite{GelDalTsy:OAVONCDG}}}
\label{esregnetdurchdieDecke}
Let $(\cX_\bullet,b,B)$ be a mixed complex, and let $(\mathcal{G}^\bullet,\delta,\{\cdot,\cdot\}, \smallsmile)$ be both 
a dg associative algebra and a dg Lie algebra (with degree shifted by one)  such that its cohomology $H^\bullet(\mathcal{G},\gd)$ is a Gerstenhaber algebra (which one may refer to as homotopy Gerstenhaber algebra). The mixed complex $\cX$ is called a {\em homotopy Gerstenhaber module} over $\mathcal{G}$ if 
 $(\cX_{-\bullet}, b)$ is both a dg module over $(\mathcal{G}^\bullet, \smallsmile, \iota)$ and a dg Lie algebra module over $(\mathcal{G}^\bullet[1], \{\cdot, \cdot\}, \cL)$ by means of two respective actions
$$
 \iota \colon \mathcal{G}^p\otimes \cX_n\to \cX_{n-p}, \quad  \cL \colon \mathcal{G}^p \otimes \cX_n \to \cX_{n-p+1}, 
 $$
 called {\em cap product} (or {\em contraction}) and \emph{Lie derivative}, respectively, such that, writing $\iota_\gvf := \iota(\gvf \otimes \cdot)$ for $\gvf \in \mathcal{G}$ and similarly for all operators in the sequel, the
 {\em Gelfan'd-Daletski\u\i-Tsygan homotopy} formula
 \begin{equation}
\label{panem}
 \!\!\!
 \begin{array}{rcl}
   [\iota_\gvf, \cL_\psi] - \iota_{\{\gvf, \psi\}} &=& [b, \cT_{\gvf,\psi}] - \cT_{\gd \gvf, \psi} -(-1)^\gvf \cT_{\gvf, \delta \psi}
 \end{array}
 \end{equation}
 holds, where $\cT$ is an operator
  $
 \cT \colon \mathcal{G}^p
 \otimes \mathcal{G}^q \otimes \cX_n\to \cX_{n+p+q-2}.
 $
 A homotopy Gerstenhaber module $\cX$ is called {\em homotopy Batalin-Vilkoviski\u\i\ (BV) module} over $\mathcal{G}$ if there is an additional operator
 $\cS \colon \mathcal{G}^p\otimes \cX_n\to \cX_{n+p-2}$ such that the  {\em Cartan-Rinehart homotopy} formul\ae
  \begin{equation}
\label{dellera1}
\begin{cases}
    \begin{array}{rcl}
\cL_\gvf &=& [B, \iota_\gvf]+[b, \cS_\gvf] + S_{\gd \gvf},\\
[\cS_\gvf, \cL_\psi] - \cS_{\{\gvf,\psi\}} &=& [B, \cT_{\gvf,\psi}]
\end{array}
\end{cases}
\end{equation}
are verified.
A Gerstenhaber resp.\ BV module is then defined by analogous relations that would hold on homology $H_\bullet(M,b)$ and cohomology $H^\bullet(\mathcal{G},\delta)$ (which then becomes a true Gerstenhaber algebra) setting all homotopy terms to zero. For example, in case of a BV module one has the following relations:
\begin{equation}
  \label{true}
\iota_{\gvf \smallsmile \psi} = \iota_\gvf \iota_\psi, \quad \cL_{\{\gvf,\psi\}} = [\cL_\gvf, \cL_\psi], \quad [\iota_\gvf, \cL_\psi] = \iota_{\{\gvf, \psi\}}, \quad  \cL_\gvf = [B, \iota_\gvf].
\end{equation}
Inspired by the obvious resemblance of these identities with the well-known ones in differential geometry, a BV module structure is also called a {\em noncommutative differential calculus} in \cite{Tsy:CH} and a {\em noncommutative Cartan calculus} in \cite{FioKow:HBOCANCC}; one might also want to call this a {\em Tamarkin-Tsygan calculus} since these structures are analysed in detail in \cite{TamTsy:NCDCHBVAAFC}. In this spirit, one may equally speak of a {\em homotopy noncommutative differential/Cartan calculus} or simply a {\em homotopy calculus} on the pair $(\mathcal{G}, \cX)$ instead of a homotopy BV module $\cX$ over $\mathcal{G}$.



\subsection{Noncommutative calculi arising from cyclic opposite modules over operads}
As already hinted at,
there is an intimate relationship between homotopy noncommutative calculi
and cyclic opposite modules 
which is expressed in the following theorem from \cite[Thm.~5.4]{Kow:GABVSOMOO}):


\begin{theorem}
  \label{terzamissione}
\!\! The structure of a cyclic unital opposite module  $(\cM, t)$ over an operad with multiplication $(\cO, \mu, e)$ induces a homotopy calculus on the pair $(\cO, \cM)$ of $k$-modules. 
  \end{theorem}

For later use, we will give some explicit formul\ae, see \cite{Kow:GABVSOMOO} and also \cite[\S6]{FioKow:HBOCANCC}. For $\gvf \in \cO(p)$, $\psi \in \cO(q)$, and $x \in \cM(n)$ one obtains
\begin{equation}
  \label{Regenregenregen}
  \begin{array}{rcl}
    B x &=& \sum\limits^n_{i=0} (-1)^{in} e \bullet_0 t^i(x),
    \\
\iota_\gvf x &=& (\mu \circ_2 \gvf) \bullet_0 x,    
\\
\cL_\gvf x &=& \sum\limits^{n-p+1}_{i=1}  (-1)^{(p-1)(i-1)} \gvf \bullet_i x + 
        \sum\limits^{p}_{i=1} (-1)^{n(i-1) + p - 1} \gvf \bullet_0 t^{i-1} (x),
\\
\cS_\gvf x &=& 
        \sum\limits^{n-p+1}_{j=1} \, 
        \sum\limits^{n - p+1}_{i=j} (-1)^{ n(j-1) + (p-1)(i-1)} e \bullet_0 \big(\gvf \bullet_i t^{j-1}(x)\big),
\\
        \cT_{\gvf,\psi}(x) &=& \sum\limits^{p-1}_{j=1}  \sum\limits^{p-1}_{i=j} (-1)^{n(j-1) + (q-1)(i-j) + p} (\gvf \circ_{p-i+j} \psi) \bullet_0 t^{j-1}(x).
  \end{array}
\end{equation}
Observe the formal analogy between the cap product $\iota_\gvf x =: \gvf \smallfrown x$ and the cup product $\gvf \smallsmile \psi = (\mu \circ_2 \psi) \circ_1 \gvf$ in the operad $\cO$.
With these explicit expressions, it is an essentially direct (but not-so-straightforward) check that on the normalised complex $\widebar{\cM}$ and for elements in $\widebar{\cO}$, the homotopy formul\ae \ \eqref{panem} and \eqref{dellera1} hold.

Combining this general fact with the statement of
Theorem \ref{plebiscito} above,
one then at once obtains in our more specific situation the main result in this section:
 
\begin{cor}
  \label{tarrega}
   The couple $(\cO, \cM)$ as defined in Eq.~\eqref{maisondumonde} can be equipped with the structure of a homotopy Cartan calculus if $M$ is a stable aYD contramodule over $U$.
In particular, this induces the structure of a BV module on
$\Coext^U_\bullet(A, M)$
over
$\Cotor^\bullet_U(A,A)$.
 \end{cor}

 With the help of Eqs.~\eqref{Regenregenregen}, 
 we can then explicitly obtain the operations that define the calculus structure, {\em cf.}~\S\ref{esregnetdurchdieDecke}. For example,
 for the cap product $\iota$, the cyclic coboundary $B$, and the Lie derivative $\cL$, a not really quick computation using Eqs.~\eqref{nawas1}, \eqref{lakritz3}, \eqref{maxdudler2}, \eqref{soschlapp}, \eqref{viatoledo1}, \eqref{extraop}, and basically all of the identities \eqref{Sch1}--\eqref{mampf3} yields:
 \begin{small}
   \begin{eqnarray}
     \nonumber
    (B f)(v^1|\ldots|v^{n+1}) &\!\!\!\!\!
    =&\!\!\!\!\!
    \textstyle\sum\limits^{n+1}_{i=1} (-1)^{(i-1)n} \gamma\Big(\!\smadotp (v^i \mpact f)\big(\smadotm \mancino (v^{i+1} | \ldots | v^{n+1}) | v^1|\ldots|v^{i-1}\big)\!\Big)                    
  \\
      \label{psytrance2}
(\iota_w f)(v^1|\ldots|v^{n-p}) &\!\!\!\!\!=&\!\!\!\!\! \gamma\big(\smadotp f(\smadotm \mancino (u^1 | \ldots | u^p) | v^1|\ldots|v^{n-p})\big) 
\\[4pt]
\nonumber
    (\cL_w f)(v^1|\ldots|v^{n-p+1})&\!\!\!\!\!=&\!\!\!\!\!
\\
\nonumber
    \label{psytrance3}
    &&
    \hspace*{-3.38cm}
     \textstyle\sum\limits^{n-p+1}_{i=1} (-1)^{(p-1)(i-1)} 
 f\big(v^1 | \ldots | v^{i-1} | v^i \mancino (u^1 | \ldots |u^p)|
 v^{i+1} | \ldots | v^{n-p+1}\big)
 \\
 \nonumber
    &&
  \hspace*{-3.5cm}
  + \textstyle\sum\limits^{p}_{i=1} (-1)^{n(i-1) + p-1}
   \gamma\Big(\!\smadotp (u^i \mpact f)\big(\smadotm \mancino (u^{i+1} | \ldots | u^{p}) |v^1|\ldots|v^{n-p+1}| u^1|\ldots|u^{i-1}\big)\!\Big) 
  \end{eqnarray}
 \end{small}
for $w := (u^1| \ldots | u^p) \in \cO(p)$, $f \in \cM(n)$, and $(v^1|\ldots|v^k) \in U^{\otimes_A k}$.
Here, if $i < j$ appears in a sum, an element $(u^j| \ldots| u^i)$ has to be read as $1_A$: for example, in the cyclic boundary $B$ the first and the last term have to be read as $ \gamma\big((\sma{-} v^1 \mpact f)\big(v^{2} | \ldots | v^{n+1}\big)\big)$ resp.\ $\gamma\big(\sma{-} (v^{n+1} \mpact f)\big(v^{1} | \ldots | v^{n}\big)\big)$, and similarly in the expression for the Lie derivative.

We spare the reader at this point to be confronted with the explicit expressions of the homotopy operators $\cS$ and $\cT$ as these explicit expressions will not be needed in what follows.

\subsection{The homotopy calculus structure for cocommutative bialgebroids}
\label{nirvana}

Cocommutative bialgebroids constitute an important class of examples of bialgebroids and will in particular cover the main example dealt with in the next section. 

In a cocommutative bialgebroid $(U,A)$, the base algebra $A$ is necessarily commutative and the source map equals the target one. This, in turn, implies that there exceptionally exists a trivial contraaction as discussed in Example \ref{trivial}: any right $A$-module $M$ is a right $U$-contramodule by means of $\Hom_A(U,M) \to M$, $f \mapsto f(1)$; if the right $A$-module $M$ also happens to be a left $U$-module, it automatically becomes a stable aYD contramodule over $U$, that is, fulfils Eqs.~\eqref{romaedintorni}--\eqref{stablehalt} as one quickly verifies by Eqs.~\eqref{Sch4} and \eqref{Sch2}.
Using the trivial contraaction notably simplifies the structure maps of the cocyclic $k$-module $C_\bullet(U,M)$ from Lemma \ref{lakritz2}, which
for any $f \in C_n(U,M)$
now become
\begin{eqnarray}
  \label{lakritz4}
(d_i f)(u^1|\ldots|u^{n-1}) \!\!\!\!&=\!\!\!\!& \left\{\!\!\!
\begin{array}{l} 
f(1 |u^1 |\ldots |u^{n-1})
\\ 
 f(u^1 |\ldots |\gD u^i |\ldots | u^{n-1})
\\
f(u^1|\ldots|u^{n-1} |1)
\end{array}\right.  
  \begin{array}{l} \mbox{if} \ i=0, \\ \mbox{if} \
  1 \leq i \leq n-1, \\ \mbox{if} \ i = n,  \end{array} 
  \\[2pt]
  \nonumber
  \label{lakritz5}
 (s_j f)(u^1 | \ldots | u^{n+1}) \!\!\!\! 
  &=\!\!\!\!& f(u^1 | \ldots | \gve(u^{j+1}) | \ldots | u^{n+1}) \ \, \quad\mbox{for} \ 0 \leq j \leq n,
  \\[2pt]
  \nonumber
    \label{lakritz6}
(t f)(u^1 | \ldots | u^n) \!\!\!\!&=\!\!\!\!& (u^1 \mpact
  f)(u^2 | \ldots | u^n | 1).
\end{eqnarray}
Observe that in this situation $D_\bullet(U,M)$ and $C_\bullet(U,M)$ are not only isomorphic as complexes but equal, that is, $b = d$, as seen from Eq.~\eqref{landliebekirsch} and Eq.~\eqref{lakritz1}.
Note that now
\begin{equation}
  \label{wasserbombe}
  d = \Hom_A(\partial, M)
\end{equation}
in case the left and right $U$-comodules in \eqref{landliebecotor} are given by $A$ itself, where $\partial$ is the differential of the cochain complex $\coc^\bullet(U,A)$ computing $\Cotor^\bullet_U(A,A)$.

Furthermore, if a cocommutative left bialgebroid is left Hopf, it is automatically right Hopf as well since the distinction between the two structures vanishes: one has $u_+ \otimes_\Aopp u_- = u_{[+]} \otimes_A u_{[-]}$ for any $u \in U$, and therefore also $u \pmact f = u \mpact f$.
On top, the trivial contraaction notably entangles the calculus operators from Eq.~\eqref{psytrance2}, which reduce to
\begin{small}
   \begin{eqnarray}
    \label{psytrance4}
    (B f)(v^1|\ldots|v^{n+1}) &\!\!\!\!\!
    =&\!\!\!\!\!
    \textstyle\sum\limits^{n+1}_{i=1} (-1)^{(i-1)n} (v^i \mpact f)(v^{i+1} | \ldots | v^{n+1} | v^1|\ldots|v^{i-1})                    
  \\
      \label{psytrance5}
(\iota_w f)(v^1|\ldots|v^{n-p}) &\!\!\!\!\!=&\!\!\!\!\! f(u^1 | \ldots | u^p | v^1|\ldots|v^{n-p}) 
\\[4pt]
\nonumber
    (\cL_w f)(v^1|\ldots|v^{n-p+1})&\!\!\!\!\!=&\!\!\!\!\!
\\
    \label{psytrance6}
    &&
    \hspace*{-3.38cm}
     \textstyle\sum\limits^{n-p+1}_{i=1} (-1)^{(p-1)(i-1)} 
 f\big(v^1 | \ldots | v^{i-1} | v^i \mancino (u^1 | \ldots |u^p)|
 v^{i+1} | \ldots | v^{n-p+1}\big)
 \\
 \nonumber
    &&
  \hspace*{-3.5cm}
  + \textstyle\sum\limits^{p}_{i=1} (-1)^{n(i-1) + p-1}
  (u^i \mpact f)(u^{i+1} | \ldots | u^{p} |v^1|\ldots|v^{n-p+1}| u^1|\ldots|u^{i-1}), 
  \end{eqnarray}
\end{small}
for $w = (u^1 | \ldots | u^p)$.
In particular, $\iota$ now becomes a simple insertion of $w$ into $f$ resp.\ literally a contraction of $f$ by the element $w$.

\begin{example}
As we saw in \S\ref{voila}, the universal enveloping algebra $V\!L$ of a Lie-Rinehart algebra $(A,L)$ is a cocommutative left bialgebroid resp.\ left (and right) Hopf algebroid and hence the considerations made in this subsection apply. In the next section, we will develop this example in detail and use Eqs.~\eqref{psytrance4}--\eqref{psytrance6} to recover the classical Cartan calculus.
  \end{example}

 \section{Example: Cartan calculi in differential geometry}
 \label{examples}

 We already briefly mentioned that the noncommutative calculus on $\Coext$ and $\Cotor$ contains the classical Cartan calculus known from differential geometry as an example in a natural way. In a more restricted context, this was already achieved in \cite{KowKra:BVSOEAT} by a calculus on $\Ext$ and $\Tor$ which, however, passes through a sort of double dual, and as a consequence requires a certain finiteness condition, the use of topological tensor products as well as completions. As we will explain now, the calculus structure obtained in the previous section applied to the special case of differential geometry does not ask for anything of all that and therefore yields a much more direct and even more general approach as one can start from Lie-Rinehart algebras
 of possibly infinite dimension.

\subsection{Lie-Rinehart algebras and classical Cartan calculus}

Let $(A,L)$ be a Lie-Rinehart algebra
with $L$ not necessarily finitely generated as an $A$-module.
We call the elements of the exterior algebra $\textstyle\bigwedge^\bullet_A \!L$ over $A$ {\em multivector fields}. The triple $(\textstyle\bigwedge^\bullet_A \!L, 0, [\cdot,\cdot]_{SN})$ defines a dg Lie algebra with respect to the zero differential along with the Schouten-Nijenhuis bracket $[\cdot, \cdot]_{SN}$ over $A$, and a Gerstenhaber algebra if we add the wedge product. Let $M$ then be both an $A$-module and a left $L$-module, where the two actions do not commute but rather reflect the presence of the anchor map, which is equivalent to saying that $M$ is a $V\!L$-module, where $V\!L$ denotes the universal enveloping algebra of $(A,L)$. The dual space
$\Hom_A(\textstyle\bigwedge_A^\bullet \!L,M)$
of alternating $M$-valued $A$-{\em multilinear forms} constitutes a mixed complex $\big(\!\Hom_A(\textstyle\bigwedge_A^\bullet \!L,M), 0, d_{\sf dR}\big)$, where
$$
d_{\sf dR} : \Hom_A(\textstyle\bigwedge^n_A \!L, M) \to \Hom_A(\textstyle\bigwedge^{n+1}_A \!L, M)
$$
is the {\em de Rham-Chevalley-Eilenberg} differential
\begin{small}
\begin{equation}
\label{fassbinder}
\begin{split}
d_{\sf dR} \go(X^1, \ldots, X^{n+1}) &:= \textstyle\sum\limits^{n+1}_{i=1} (-1)^{i-1} X^i\big(\go(X^1, \ldots,  \hat{X}^i, \ldots, X^{n+1})\big) \\
&\quad + \textstyle\sum\limits_{i < j} (-1)^{i+j-1} \go([X^i, X^j], X^1, \ldots, \hat{X}^i, \ldots, \hat{X}^j, \ldots,  X^{n+1}),
\end{split}
\end{equation}
\end{small}
where as usual $\hat X^i$ means omission. With respect to this differential,
we can enhance the statement in Lemma \ref{dumdidum} that 
the HKR map
$\mathrm{Alt}:  \textstyle\bigwedge_A^n \!L \to V\!L^{\otimes^{ll}_A n}$ from Eq.~\eqref{hkr} induces a morphism of chain complexes. To begin with, observe that in this case, we have $\big(D_\bullet(V\!L, M), b\big) = \big(C_\bullet(V\!L, M), d\big)$, as explained in \S\ref{nirvana}.

\begin{lem}
  Let $(A,L)$ be a Lie-Rinehart algebra with $L$ flat but not necessarily finitely generated as an $A$-module and $M$ a left $V\!L$-module. Then the pull-back of the HKR map $\mathrm{Alt}$ yields a morphism
  $$
  \big(C_\bullet(V\!L,M), d, B\big) \lra  \big(\!\Hom_A(\textstyle\bigwedge^\bullet_A \!L, M), 0, d_{\sf dR} \big)
  $$
  of mixed complexes. In particular,
   \begin{equation}
    \label{oskar1}
    \Hom_A({\Alt},M) \circ B = d_{\sf dR} \circ \Hom_A({\Alt},M)
   \end{equation}
   holds.
  \end{lem}

\begin{proof}
 Thanks to Lemma \ref{dumdidum}, we are already left with showing Eq.~\eqref{oskar1}.
  Indeed, for any $f \in C_{n-1}(V\!L,A)$, using \eqref{psytrance4}, \eqref{mpaction}, and \eqref{regalate} one has
  \begin{small}
  \begin{equation*}
    \begin{split}
    &  ((B f)  \circ \Alt)(X^1 \wedge \cdots \wedge X^{n})
\\
&=
    \textstyle\frac{1}{n!}\sum\limits_{\gs \in S(n)}(-1)^\gs \sum\limits^{n}_{i=1} (-1)^{(i-1)n} (X^{\gs(i)} \mpact f)(X^{\gs(i+1)} | \ldots | X^{\gs(n)} | X^{\gs(1)}|\ldots|X^{\gs(i-1)})  
\\
 &=
 \textstyle\frac{n}{n!}\sum\limits_{\gs \in S(n)}(-1)^\gs (X^{\gs(1)}
 \mpact f)(X^{\gs(2)} |\ldots|X^{\gs(n)})
 \\
 &=
 \textstyle\frac{1}{(n-1)!}\sum\limits_{\gs \in S(n)}(-1)^\gs X^{\gs(1)}
 f(X^{\gs(2)} |\ldots|X^{\gs(n)})
 \\
&
 \quad - \textstyle\frac{1}{(n-1)!}\sum\limits_{\gs \in S(n)}(-1)^\gs \sum\limits^{n-1}_{i=1} f(X^{\gs(1)} |\ldots|X^{\gs(i)} X^{\gs(i+1)} | \ldots | X^{\gs(n)})
 \\
 &=
 \textstyle\frac{1}{(n-1)!}\sum\limits_{\tau \in S(n-1)}(-1)^\tau \sum\limits^{n}_{i=1} (-1)^{i-1} X^i
 \big(f(X^{\gs(1)} |\ldots|X^{\gs(i-1)}|X^{\gs(i+1)} |\ldots|X^{\gs(n)})\big)
 \\
&
 \quad - \textstyle\frac{1}{(n-1)!}\sum\limits_{\gs \in S(n)}(-1)^\gs \sum\limits^{n-1}_{i=1} f(X^{\gs(1)} |\ldots|X^{\gs(i)} X^{\gs(i+1)} | \ldots | X^{\gs(n)})
 \\
 &=     \textstyle\sum\limits^{n}_{i=1} (-1)^{i-1} X^i\big((f \circ \Alt)(X^1, \ldots,  \hat{X}^i, \ldots, X^n)\big) \\
&\quad + \textstyle\sum\limits_{i < j} (-1)^{i+j-1} (f \circ \Alt)([X^i, X^j], X^1, \ldots, \hat{X}^i, \ldots, \hat{X}^j, \ldots,  X^{n})
\\
&=    d_{\sf dR} (f \circ \Alt)(X^1, \ldots, X^{n}), 
  \end{split}
  \end{equation*}
  \end{small}
\!\!\! using $\Alt([X,Y]) = XY - YX$, where on the right hand side $X, Y$
  are seen as elements~in $V\!L$, as $\Alt$ in degree $1$ is simply the canonical map $L \to V\!L$. This concludes the proof.
  \end{proof}

Consider now the left inverse
$
\cP:  
V\!L^{\otimes^{ll}_A n}
\to
\textstyle\bigwedge_A^n \!L 
$
of $\Alt$ given by
$
\cP := \pr \wedge \cdots \wedge \pr,
$
as in \cite[Prop.~XVIII.7.6]{Kas:QG},
where $\pr: V\!L \to L$ denotes the natural projection; this is a quasi-isomorphism as well. Using our general construction in Corollary \ref{tarrega}, we can then deduce calculus operations on multilinear forms over multivector fields:

\begin{prop}
  \label{nyc}
  Let $(A,L)$ be a Lie-Rinehart algebra with $L$ flat but not necessarily finitely generated as an $A$-module and $M$ a left $V\!L$-module. By means of
  \begin{eqnarray}
     \label{bibo2}
\iota_{(-)} &:=&   \Hom_A({\Alt},M) \circ \iota_{\Alt(-)} \circ \Hom_A(\cP,M),  \\
    \label{bibo3}
\cL_{(-)} &:=& \Hom_A({\Alt},M) \circ \cL_{\Alt(-)} \circ \Hom_A(\cP,M), 
  \end{eqnarray}
 along with the differential $d_{\sf dR}$ from Eq.~\eqref{fassbinder} and the homotopy operators $\cS = \cT = 0$, one obtains a (homotopy) noncommutative calculus 
on 
the pair of the Gerstenhaber algebra  $(\textstyle\bigwedge^\bullet_A \!L, 0, [\cdot,\cdot]_{SN})$ and the mixed complex $\big(\!\Hom_A(\textstyle\bigwedge_A^\bullet \!L,M), 0, d_{\sf dR}\big)$ in the sense of \S\ref{esregnetdurchdieDecke}. Explicitly, Eqs.~\eqref{bibo2} and \eqref{bibo3} result into maps
\begin{small}
  \begin{eqnarray*}
    \iota: \textstyle\bigwedge^p_A \!L \otimes \Hom_A(\textstyle\bigwedge^n_A \!L, M)  &\to&  \Hom_A(\textstyle\bigwedge^{n-p}_A \!L, M) , \\ \iota_Y\go(X^1, \ldots, X^{n-p}) &=& \go(Y^1, \ldots,Y^p, X^1, \ldots, X^{n-p}),
    \\[3pt]
 \cL: \textstyle\bigwedge^p_A \!L \otimes \Hom_A(\textstyle\bigwedge^n_A \!L, M)  &\to&  \Hom_A(\textstyle\bigwedge^{n-p+1}_A \!L, M),
 \\
\cL_Y\go(X^1, \ldots, X^{n-p+1}) &=&
\textstyle\sum\limits^p_{i=1} (-1)^{i-1} Y^i \big(\go(Y^1, \ldots, \hat Y^i, \ldots, Y^p, X^1, \ldots, X^{n-p+1})\big)
\\
&& \hspace*{-2.5cm}
+ \textstyle\sum\limits_{j=1}^{p}\sum\limits^{n-p+1}_{i=1} (-1)^j \go(Y^1, \ldots, \hat Y^j, \ldots, Y^p, X^1, \ldots, [Y^j, X^i] , \ldots,  X^{n-p+1})
\\
&& \hspace*{-2.5cm}
+
\textstyle\sum\limits_{i < j} (-1)^{i+j-1} \go([Y^i, Y^j], Y^1, \ldots, \hat{Y}^i, \ldots, \hat{Y}^j, \ldots,  Y^p, X^1, \ldots, X^{n-p+1}),
 \end{eqnarray*}
\end{small}
   for a multivector field $Y := Y^1 \wedge \cdots \wedge Y^p$,
and hence coincide with the classical operations of contraction
   and Lie derivative from \cite{Rin:DFOGCA}.
\end{prop}

In case $(A, L) \!=\! (C^\infty(Q), \gG(E))$ arises from a Lie algebroid $E\! \to\! Q$ over a smooth manifold
 $Q$, the above calculus is the one given by $E$-differential
 forms and $E$-multivector fields (as detailed in \cite[\S18]{CanWei:GMFNCA}), and if the Lie algebroid is given by the tangent bundle $TQ$,
 this yields the well-known Cartan calculus in differential geometry \cite{Car:LTDUGDLEDUEFP}.

 \begin{proof}[Proof of Proposition \ref{nyc}]
   That $\iota$ defined as in \eqref{bibo2} has the stated explicit form is quickly checked as follows: for two multivector fields $Y = Y^1 \wedge \cdots \wedge Y^p$ and $X = X^1 \wedge \cdots \wedge X^{n-p}$  as well as a form $\go \in \Hom_A(\textstyle\bigwedge_A^\bullet \!L,M)$, one has using \eqref{psytrance5}
 \begin{equation*}
    \begin{split}
      \iota_Y\go (X) &=  \iota_{\Alt(Y)}(\go \circ \cP)(\Alt(X))
      \\
      &= \textstyle\frac{1}{p!(n-p)!} \sum_{\gs, \tau} (-1)^{\gs + \tau}
      (\go \circ P) (Y^{\gs(1)} | \ldots| Y^{\gs(p)}| X^{\tau(1)}| \ldots | X^{\tau(n-p)} )
\\
      &= \go(Y^1 , \ldots,  Y^p, X^1, \ldots, X^{n-p} ).
    \end{split}
    \end{equation*}
 That $\cL$ as defined in \eqref{bibo3} also has the stated explicit form is much more laborious but follows from Eqs.~\eqref{psytrance6} and \eqref{regalate} by a computation very similar to the one that proved \eqref{oskar1}, which is why we omit it. That the so-defined operators define a homotopy calculus is of course known; more precisely, the (co)simplicial differentials being zero, this even furnishes a calculus fulfilling the customary identities \eqref{true}. More elegantly, the calculus identities can be directly derived from our general result in Theorem \ref{terzamissione} taking into account  that also $\cP$ can be promoted to a morphism of mixed complexes, that is,
 \begin{equation}
    \label{this}
    B \circ \Hom_A(\cP,M) = \Hom_A(\cP,M) \circ d_{\sf dR},
 \end{equation}
 holds as well,
 which is again proven by a straightforward computation along the lines of the proof of \eqref{oskar1} and will therefore be skipped.
One then has
\begin{eqnarray*}
  \cL_Y 
  &
\!\!\!\!\!\!
\overset{{\scriptscriptstyle{\eqref{bibo3}}}}{=}
&
\!\!\!\!\!\!
\Hom_A({\Alt},M) \circ \cL_{\Alt(Y)} \circ \Hom_A(\cP,M)
\\
  &
\!\!\!\!\!\!
\overset{{\scriptscriptstyle{\eqref{true}}}}{=}
&
\!\!\!\!\!\!
\Hom_A({\Alt},M) \circ \big(B \circ \iota_{\Alt(Y)} + (-1)^{\deg Y} \iota_{\Alt(Y)} \circ B\big) \circ \Hom_A(\cP,M)
\\
  &
\!\!\!\!\!\!
\overset{{\scriptscriptstyle{\eqref{oskar1}, \eqref{this}}}}{=}
&
\!\!\!\!\!\!
d_{\sf dR} \circ \Hom_A({\Alt},M) \circ \iota_{\Alt(Y)} \circ \Hom_A(\cP,M)
\\
&&
+ (-1)^{\deg Y}
\Hom_A({\Alt},M) \circ 
\iota_{\Alt(Y)} \circ \Hom_A(\cP,M) \circ d_{\sf dR}
\\
  &
\!\!\!\!\!\!
\overset{{\scriptscriptstyle{\eqref{bibo2}}}}{=}
&
\!\!\!\!\!\!
    [d_{\sf dR}, \iota_Y],
    \end{eqnarray*}
which concludes the proof.
   \end{proof}

    By functoriality, in homology the map $\cP$ also becomes a right inverse of $\Alt$ and from Eqs.~\eqref{oskar1}--\eqref{bibo3} immediately follows:

\begin{theorem}
  \label{transactions}
  Let $(A,L)$ be a Lie-Rinehart algebra with $L$ projective but not necessarily finitely generated as an $A$-module and $M$ an $A$-injective left $V\!L$-module. Then the HKR map $\mathrm{Alt}$ induces an isomorphism of BV modules (or noncommutative differential calculi) between $\big(\!\textstyle\bigwedge^\bullet_A \!L, \Hom_A(\textstyle\bigwedge^\bullet_A \!L, M)\big)$  and $\big(\!\Cotor^\bullet_{V\!L}(A,A), \Coext^{V\!L}_\bullet(A,M)\big)$. In particular, $\mathrm{Alt}$ (resp.\ its pull-back)
  commutes with all possible calculus operators, that is,   
  \begin{eqnarray}
   \label{oskar1a}
   \Hom_A({\Alt},M) \circ B &=& d_{\sf dR} \circ \Hom_A({\Alt},M), \\
    \label{oskar2}
  \Hom_A({\Alt},M) \circ \iota_{\Alt(-)} &=& \iota_{(-)} \circ \Hom_A({\Alt},M),  \\
    \label{oskar3}
 \Hom_A({\Alt},M) \circ \cL_{\Alt(-)} &=& \cL_{(-)} \circ \Hom_A({\Alt},M), 
  \end{eqnarray}
  when descending to homology.
  \end{theorem}

\subsection{Lie-Rinehart cohomology and cyclic homology}
We close this example section by a few words on the relation between Lie-Rinehart cohomology and cyclic cohomology. Ignoring the zero differential in the mixed complex $\big(\!\Hom_A(\textstyle\bigwedge_A^\bullet \!L,M), 0, d_{\sf dR}\big)$, one obtains a Chevalley-Eilenberg cochain complex that generalises the classical complex computing Lie algebra cohomology:

\begin{dfn}{\cite{Rin:DFOGCA}}
The cohomology of the cochain complex $\big(\!\Hom_A(\textstyle\bigwedge_A^\bullet \!L,M), d_{\sf dR}\big)$  denoted by $H^\bullet(L,M)$ is called the {\em Lie-Rinehart cohomology} (with values in $M$) of $(A,L)$.
  \end{dfn}

Dually to Lemma \ref{dumdidum}, it is a well-known fact that if $L$ is $A$-projective, then the Lie-Rinehart cohomology is an $\Ext$-group again, see \cite[\S4]{Rin:DFOGCA}. More precisely, together with Eq.~\eqref{erkaeltetmalwieder} we have
$$
\Ext_{V\!L}^\bullet(A,M) \simeq H^\bullet(L, M), \qquad \Coext^{V\!L}_\bullet(A,M) \simeq \Hom_A(\textstyle\bigwedge^\bullet_A \!L, M),
$$
which allows us to state:

\begin{prop}
  Let $(A,L)$ be a Lie-Rinehart algebra, where $L$ is projective but not necessarily finite as an $A$-module, and $M$ a left $V\!L$-module. Then the HKR map \eqref{hkr} induces the isomorphisms
\begin{footnotesize}
  \begin{equation*}
    \begin{array}{rcl}
      HC_n(V\!L, M) &\!\!\!\!\!\simeq&\!\!\!\!\! \Hom_A(\textstyle\bigwedge^n_A \!L, M)/d_{\sf dR}\big(\!\Hom_A(\textstyle\bigwedge^{n-1}_A \!L, M) \!\big) \oplus H^{n-2}(L, M) \oplus H^{n-4}(L, M) \oplus \cdots,
      \\[3pt]
HC^n(V\!L, M) &\!\!\!\!\!\simeq&\!\!\!\!\! H^{n}(L, M) \oplus H^{n-2}(L, M) \oplus \cdots,
      \end{array}
  \end{equation*}
  \end{footnotesize}
where $HC_\bullet$ denotes the cyclic homology defined by the cyclic module \eqref{lakritz4}, and $HC^\bullet$ the cyclic cohomology with respect to the cocyclic module \eqref{nadennwommama} for the trivial contraaction.
  \end{prop}

\begin{proof}
  The first isomorphism follows from Lemma \ref{dumdidum} together with Eq.~\eqref{oskar1} by computing the total homology of the trivial mixed complex  $\big(\!\Hom_A(\textstyle\bigwedge_A^\bullet \!L,M), 0, d_{\sf dR}\big)$. The second isomorphism follows from the fact that the cyclic boundary $B$ associated to the cocyclic module \eqref{nadennwommama} for the trivial contraaction in case of a cocommutative bialgebroid induces the zero map in the cohomology of the columns of the respective mixed complex, which can be either computed directly or obtained by simply $A$-linearly dualising \cite[Thm.~2.16 \& 3.14]{KowPos:TCTOHA}.  
  \end{proof}

\appendix

\section{The cyclic category}

\subsection{Cocyclic and para-cocyclic modules}

%

Recall from, {\em e.g.}, \cite[\S6.1]{Lod:CH} that a cyclic
$k$-module is a simplicial
$k$-module $(C_\bullet,d_\bullet,s_\bullet)$ 
resp.\ a cocyclic $k$-module is a
cosimplicial $k$-module
$(C^\bullet,\delta_\bullet,\sigma_\bullet)$
together with $k$-linear maps 
$t: C_n \rightarrow C_n$ resp.\
$\tau: C^n \rightarrow C^n$ in degree $n$,
satisfying, respectively
\begin{equation}
\label{belleville}
\!\!\!\!\!\!\!\!
\begin{array}{cc}
\begin{array}{rcl}
d_i \circ t  &\!\!\!\!\!\!=&\!\!\!\!\!\! \left\{\!\!\!
\begin{array}{ll}
t \circ d_{i-1} 
& \!\!\!\! \mbox{if} \ 1 \leq i \leq n, \\
 d_n & \!\!\!\! \mbox{if} \
i = 0,
\end{array}\right.
\\
\
\\
s_i \circ t &\!\!\!\!\!\!=&\!\!\!\!\!\! \left\{\!\!\!
\begin{array}{ll}
t \circ s_{i-1} & \!\!\!\!
\mbox{if} \ 1 \leq i \leq n, \\
 t^2 \circ s_n
 & \!\!\!\! \mbox{if} \
 i = 0,
\end{array}\right.
\\
\
\\
 t^{n+1} &\!\!\!\!\!=&\!\!\!\!\! \id_{C_n},
\end{array}
\!\!\!\!&\!\!\!\!
\begin{array}{rcll}
\tau \circ \gd_i &\!\!\!\!\!=&\!\!\!\!\! \left\{\!\!\!
\begin{array}{l}
\gd_{i-1}\circ \tau \\
 \gd_n 
\end{array}\right. & \!\!\!\!\!\!\!\!\! \begin{array}{l} \mbox{if} \ 1 \leq i \leq n,
 \\ \mbox{if} \ i = 0,
\end{array}
\\
\
\\
\tau \circ \sigma_i &\!\!\!\!\!=&\!\!\!\!\!
\left\{\!\!\!
\begin{array}{l}
\sigma_{i-1} \circ \tau \\
 \sigma_n \circ \tau^2 
\end{array}\right. & \!\!\!\!\!\!\!\!\!
\begin{array}{l} \mbox{if} \ 1 \leq i
 \leq n,  
 \\ \mbox{if} \ i = 0, \end{array}
\\
\
\\
 \tau^{n+1} &\!\!\!\!\!=&\!\!\!\!\! \id_{C^n}.
\end{array}
\end{array}
\end{equation}
In the definition of a {\em para-cyclic} resp.\ {\em para-cocyclic} $k$-module one drops the last identity, that is, the cyclic resp.\ cocyclic operator does not power to the identity any more.
More conceptually, 
cyclic $k$-modules resp.\ cocyclic ones can be viewed as functors 
$\gL^\op \to \kmod$ resp.\
$\gL \to  \kmod$, where $\Lambda$ is Connes' cyclic
category, see {\em loc.~cit.} for a detailed description. 
A cyclic $k$-module allows to introduce the {\em cyclic} (or {\em Connes-Rinehart-Tsygan}) {\em boundary} 
\begin{equation}
\label{e-mantra}
B := (1 - (-1)^n t) s_{-1} \cN,
\end{equation}
where $s_{-1} := t s_n$ is the {\em extra degeneracy} and $\cN := \sum^n_{i=0} (-1)^{i+n} t^i_n$ the {\em norm operator}; an analogous construction leads to the cyclic {\em co\/}boundary in case of a cocyclic $k$-module. In both cases, together with the respective (co)simplicial (co)boundary summing all (co)faces with alternating sign, this leads to a mixed complex, the total (co)homology of which defines cyclic (co)homology.

\subsection{The cyclic dual}
\label{boing}
It is a well-known fact (see \cite{Con:CCEFE} or \cite[Prop.~6.1.11]{Lod:CH}) that the cyclic category $\gL$ is self-dual, which allows to identify cocyclic $k$-modules and cyclic $k$-modules, even in infinitely many ways due to the autoequivalences of the cyclic category \cite[\S6.1.14]{Lod:CH}. The standard choice to pass from a cocyclic module $(X^\bullet, \gd_\bullet, \gs_\bullet, \tau)$ to a cyclic module  $(X_\bullet, d_\bullet, s_\bullet, t)$ is given by setting $X_n := X^n$ for all $n \in \N$ along with
\begin{equation}
\label{cyclicdual}
  d_0 := \gs_{n-1} \tau, \quad d_i := \gs_{i-1}, \quad s_j := \gd_j, \quad t := \tau^{-1}
\end{equation}
for $1 \leq i \leq n$ and $0 \leq j \leq n$. Note that in this convention the last coface $\gd_{n+1}$ is not used.

\section{Algebraic operads and (cyclic) opposite modules}
\label{pamukkale}

In this appendix, we briefly present the necessary material on operads, (cyclic) opposite modules over operads, and their connection to Gerstenhaber algebras resp.\ calculi. Modern expositions for more detailed information on operads can be found, {\em e.g.}, in \cite{LodVal:AO, Lei:HOHC}.

\subsection{Operads and Gerstenhaber algebras}
\label{pamukkale1}

A {\em non-$\gS$ operad} $\cO$ in the category $\kmod$
of $k$-modules is a sequence $\{\cO(n)\}_{n \geq 0}$ of $k$-modules 
endowed with $k$-bilinear operations $\circ_i: \cO(p) \otimes \cO(q) \to \cO({p+q-1})$ for $i = 1, \ldots, p$
subject to
\begin{eqnarray}
\label{danton}
\nonumber
\gvf \circ_i \psi &=& 0 \qquad \qquad \qquad \qquad \qquad \! \mbox{if} \ p < i \quad \mbox{or} \quad p = 0, \\
(\varphi \circ_i \psi) \circ_j \chi &=& 
\begin{cases}
(\varphi \circ_j \chi) \circ_{i+r-1} \psi \qquad \mbox{if} \  \, j < i, \\
\varphi \circ_i (\psi \circ_{j-i +1} \chi) \qquad \hspace*{1pt} \mbox{if} \ \, i \leq j < q + i, \\
(\varphi \circ_{j-q+1} \chi) \circ_{i} \psi \qquad \mbox{if} \ \, j \geq q + i.
\end{cases}
\end{eqnarray}
Call the operad {\em unital} if there is an {\em identity} $\mathbb{1} \in \cO(1)$ such that 
$
\gvf \circ_i \mathbb{1} = \mathbb{1} \circ_1 \gvf = \gvf
$ 
for all $\gvf \in \cO(p)$ and $i \leq p$, and call the operad {\em with multiplication} if there exist a {\em multiplication}  $\mu \in \cO(2)$ and a {\em unit} $e \in \cO(0)$ such that $\mu \circ_1 \mu = \mu \circ_2 \mu$ and 
$\mu \circ_1 e = \mu \circ_2 e = \mathbb{1}$. An operad with multiplication will be denoted $(\cO, \mu, e)$.
Such an object naturally defines a cosimplicial $k$-module 
given by $\cO^p := \cO(p)$ with faces and degeneracies for $\gvf \in \cO(p)$ given by $\gd_0 \gvf := \mu \circ_2 \gvf$, $\gd_i \gvf := \gvf \circ_i \mu$ for $i = 1, \ldots, p$, and $\gd_{p+1} \gvf := \mu \circ_1 \gvf$, along with $\sigma_j(\gvf) := \gvf \circ_{j+1} e$ for $j = 0, \ldots, p-1$. One obtains a cochain complex denoted by the same symbol $\cO$, with $\cO(n)$ in degree $n$, differential $\cO(n) \to \cO({n+1})$ given by $\gd := \sum^{n+1}_{i=0} (-1)^i \gd_i$, and cohomology
$
H^\bullet(\cO) := H(\cO, \gd).
$ 
Define then the {\em cup product} 
\begin{equation}
\label{nuvole}
\psi \smallsmile \gvf := (\mu \circ_2 \psi) \circ_1 \gvf \in \cO(p+q),
\end{equation}  
for $\gvf \in \cO(p)$ and $\psi \in \cO(q)$. As a consequence, $(\cO, \smallsmile, \gd)$ determines a dg algebra. One furthermore defines the {\em Gerstenhaber bracket} as 
\begin{equation}
\label{naemlichhier}
{\{} \varphi,\psi \}
:= \varphi\{\psi\} - (-1)^{(p-1)(q-1)} \psi\{\varphi\}, 
\end{equation}
where $\varphi\{\psi\} := \sum^{p}_{i=1}
(-1)^{(q-1)(i-1)} \varphi \circ_i \psi  \in \cO({p+q-1})$ is
the sum over all possible partial compositions.
Observe that $\{\mu,\mu\} = 0$ as well as 
\begin{equation}
\label{immaginedellacitta`}
\gd \varphi = (-1)^{p+1} \{\mu, \varphi \}.
\end{equation}
It is a well-known result that descending to cohomology the triple $(H^\bullet(\cO), \smallsmile, \{\cdot, \cdot\})$ constitutes a Gerstenhaber algebra \cite{Ger:TCSOAAR, McCSmi:ASODHCC}.

\subsection{(Cyclic) opposite $\cO$-modules {\cite{Kow:GABVSOMOO}}}
\label{dauerregen}
Let $\cO$ be an operad with partial composition denoted by $\circ_i$, as above.
A {\em (left) opposite $\cO$-module} is
a sequence of $k$-modules $\{ \cM(n) \}_{n \geq 0}$ together with $k$-linear
operations,  
$
        \bullet_i : 
        \cO(p) \otimes \cM(n) \to \cM({n-p+1})
$
for
$
i = 1, \ldots, n- p +1,
$
declared to be zero if $p > n$, and 
subject to   
\begin{equation}
\label{TchlesischeStr}
\gvf \bullet_i (\psi \bullet_j x) = 
\begin{cases} 
\psi \bullet_j (\gvf \bullet_{i + q - 1}  x) \quad & \mbox{if} \ j < i, 
\\
(\gvf \circ_{j-i+1} \psi) \bullet_{i}  x \quad & \mbox{if} \ j - p < i \leq j, \\
\psi \bullet_{j-p + 1} (\gvf \bullet_{i}  x) \quad & \mbox{if} \ 1 \leq i \leq j - p,
\end{cases}
\end{equation}
for $\gvf \in \cO(p)$, $\psi \in \cO(q)$, and $x \in
\cM(n)$, where $p > 0$, $q \geq 0$, $n \geq 0$
(in case $p=0$ delete the middle relation).
%
An opposite $\cO$-module is called {\em unital} if  
$
\mathbb{1} \bullet_i x = x$ for $i = 1, \ldots, n$
and all $x \in \cM(n)$.


A {\em cyclic (unital, left) opposite $\cO$-module} 
is a (unital, left) opposite $\cO$-module $\cM$ 
endowed with two additional structures:
an {\em extra} ($k$-linear) composition map
$$
\bullet_0: \cO(p) \otimes \cM(n) \to \cM({n-p+1}), \quad 0 \leq p \leq n+1,
$$
declared to be zero if $p > n+1$ such that the relations
\eqref{TchlesischeStr} and unitality are fulfilled for $i=0$ as well;
%
moreover, a degree-preserving morphism $t: \cM(n) \to \cM(n)$ for all $ n \geq 1$ with the property $t^{n+1} = \id_{\cM(n)}$ and such that
\begin{equation}
\label{lagrandebellezza1}
t(\gvf \bullet_{i} x) = \gvf \bullet_{i+1} t(x),  \qquad i = 0, \ldots, n-p,
\end{equation}
holds for $\gvf \in \cO(p)$ and $x \in \cM(n)$.

See \cite{Kow:GABVSOMOO}
or \cite{FioKow:HBOCANCC} for more information, examples, and illustrations on (cyclic) opposite $\cO$-modules (termed ``comp modules'' in the former).

%
%
%
%

A cyclic
unital opposite module 
$(\cM,t)$ 
over an operad with multiplication 
$(\cO, \mu, e)$ carries the structure of a cyclic $k$-module \cite[Prop.~3.5]{Kow:GABVSOMOO}:
the faces {$d_i\colon \cM(n) \to \cM({n-1})$ and} degeneracies $s_j\colon \cM(n) \to \cM({n+1})$ of the underlying simplicial object given by 
\begin{equation}
\label{colleoppio}
\begin{array}{rcll}
d_i(x) & = & \mu \bullet_{i} x, & i = 0, \ldots, n-1, \\
d_n(x) & = & \mu \bullet_0 t (x), & \\
s_j (x)& = & e \bullet_{j+1} x, & j = 0, \ldots, n, \\
\end{array}
\end{equation}
where $x \in \cM(n)$, can be easily shown to be compatible with the cyclic operator $t$ in the sense of Eqs.~\eqref{belleville}.
Defining 
the differential $b\colon \cM(n)\to \cM({n-1})$ by
$
b=\sum_{i=0}^n (-1)^id_i,
$
the pair $(\cM, b)$ becomes a chain complex, and by means of
$B\colon \cM(n) \to \cM(n+1)$ defined as in Eq.~\eqref{e-mantra},
%
%
the triple $(\cM, b, B)$ 
becomes a mixed (chain) complex. Here, the
extra degeneracy turns out as $s_{-1} := t \, s_n = e \bullet_0 -$, explaining the terminology {\em extra operation} for $\bullet_0$.
%
To simplify matters, we will usually work on the normalised complex $\widebar{\cM}$,
the quotient of $\cM$ by the (acyclic) subcomplex spanned by the
images of the degeneracy maps. For example, on $\widebar{\cM}$ the cyclic coboundary simplifies to
$s_{-1} \, N$, which in this case becomes explicitly
\begin{equation*}
\label{extra2}
B(x) = \textstyle\sum\limits^n_{i=0} (-1)^{in} e \bullet_0 t^i(x).
\end{equation*}
Likewise, $\widebar{\cO}$ denotes the intersection of the kernels of the codegeneracies in the cosimplicial $k$-module obtained from the operad with multiplication $(\cO, \mu, e)$.

The nice feature of
cyclic opposite $\cO$-modules and basically the reason why they were introduced is that they automatically turn into homotopy BV modules in the sense of \S\ref{esregnetdurchdieDecke}, see Theorem \ref{terzamissione} (which is Theorem 5.4 in \cite{Kow:GABVSOMOO}).

\section{Identities for left and right Hopf algebroids}
\label{nebula}
In this appendix, we gather a couple of compatibility conditions between the (inverse of the) two Hopf-Galois maps in question and the (co)product and (co)unit in a left bialgebroid. Recall that a left bialgebroid $(U, A)$ is called a left resp.\ right Hopf algebroid if
the map 
$ \alpha_\ell$ resp.\ $\ga_r$ from Eqs.~\eqref{nochmehrRegen} is a bijection in which case we write 
$
 u_+ \otimes_\Aopp u_-  :=  \alpha_\ell^{-1}(u \otimes_\ahha 1)
 $
 and
 $
   u_{[+]} \otimes_\ahha u_{[-]}  :=  \alpha_r^{-1}(1 \otimes_\ahha u),
   $
   respectively.
One then easily verifies that for a left Hopf algebroid
\begin{eqnarray}
\label{Sch1}
u_+ \otimes_\Aopp  u_- & \in
& U \times_\Aopp U,  \\
\label{Sch2}
u_{+(1)} \otimes_\ahha u_{+(2)} u_- &=& u \otimes_\ahha 1 \quad \in U_{\!\ract} \! \otimes_\ahha \! {}_\lact U,  \\
\label{Sch3}
u_{(1)+} \otimes_\Aopp u_{(1)-} u_{(2)}  &=& u \otimes_\Aopp  1 \quad \in  {}_\blact U \! \otimes_\Aopp \! U_\ract,  \\
\label{Sch4}
u_{+(1)} \otimes_\ahha u_{+(2)} \otimes_\Aopp  u_{-} &=& u_{(1)} \otimes_\ahha u_{(2)+} \otimes_\Aopp u_{(2)-},  \\
\label{Sch5}
u_+ \otimes_\Aopp  u_{-(1)} \otimes_\ahha u_{-(2)} &=&
u_{++} \otimes_\Aopp u_- \otimes_\ahha u_{+-},  \\
\label{Sch6}
(uv)_+ \otimes_\Aopp  (uv)_- &=& u_+v_+ \otimes_\Aopp v_-u_-,
\\
\label{Sch7}
u_+u_- &=& s (\varepsilon (u)),  \\
\label{Sch8}
\varepsilon(u_-) \blact u_+  &=& u,  \\
\label{Sch9}
(s (a) t (b))_+ \otimes_\Aopp  (s (a) t (b) )_-
&=& s (a) \otimes_\Aopp s (b)
\end{eqnarray}
holds, 
where in  \eqref{Sch1}  we mean 
\begin{equation*}
\label{petrarca}
   U \! \times_\Aopp \! U   :=
   \big\{ {\textstyle \sum_i} u_i \otimes v_i \in {}_\blact U  \otimes_\Aopp  U_{\!\ract} \mid {\textstyle \sum_i} u_i \ract a \otimes v_i = {\textstyle \sum_i} u_i \otimes a \blact v_i, \ \forall a \in A \big\}.
\end{equation*}
If the left bialgebroid $(U,A)$ is a right Hopf algebroid instead, one analogously obtains:
\begin{eqnarray}
\label{Tch1}
u_{[+]} \otimes_\ahha  u_{[-]} & \in
& U \times^{\scriptscriptstyle A} U,  \\
\label{Tch2}
u_{[+](1)} u_{[-]} \otimes_\ahha u_{[+](2)}  &=& 1 \otimes_\ahha u \quad \in U_{\!\ract} \! \otimes_\ahha \! {}_\lact U,  \\
\label{Tch3}
u_{(2)[-]}u_{(1)} \otimes_\ahha u_{(2)[+]}  &=& 1 \otimes_\ahha u \quad \in U_{\!\bract} \!
\otimes_\ahha \! \due U \lact {},  \\
\label{Tch4}
u_{[+](1)} \otimes_\ahha u_{[-]} \otimes_\ahha u_{[+](2)} &=& u_{(1)[+]} \otimes_\ahha
u_{(1)[-]} \otimes_\ahha  u_{(2)},  \\
\label{Tch5}
u_{[+][+]} \otimes_\ahha  u_{[+][-]} \otimes_\ahha u_{[-]} &=&
u_{[+]} \otimes_\ahha u_{[-](1)} \otimes_\ahha u_{[-](2)},  \\
\label{Tch6}
(uv)_{[+]} \otimes_\ahha (uv)_{[-]} &=& u_{[+]}v_{[+]}
\otimes_\ahha v_{[-]}u_{[-]},  \\
\label{Tch7}
u_{[+]}u_{[-]} &=& t (\varepsilon (u)),  \\
\label{Tch8}
u_{[+]} \bract \varepsilon(u_{[-]})  &=&  u,  \\
\label{Tch9}
(s (a) t (b))_{[+]} \otimes_\ahha (s (a) t (b) )_{[-]}
&=& t(b) \otimes_\ahha t(a),
\end{eqnarray}
where in  \eqref{Tch1}  we mean 
\begin{equation*}  \label{petrarca2}
   U \times^{\scriptscriptstyle A} U   :=
   \big\{ {\textstyle \sum_i} u_i \otimes  v_i \in U_{\!\bract}  \otimes_\ahha \!  \due U \lact {} \mid {\textstyle \sum_i} a \lact u_i \otimes v_i = {\textstyle \sum_i} u_i \otimes v_i \bract a,  \ \forall a \in A  \big\}.
\end{equation*}
If the left bialgebroid $(U,A)$ happens to be simultaneously a left and a right Hopf algebroid, it is an easy check that on top the {\em mixed compatibility relations}
\begin{eqnarray}
\label{mampf1}
u_{+[+]} \otimes_\Aopp u_{-} \otimes_\ahha u_{+[-]} &=& u_{[+]+} \otimes_\Aopp u_{[+]-} \otimes_\ahha u_{[-]}, \\
\label{mampf2}
u_+ \otimes_\Aopp u_{-[+]} \otimes_\ahha u_{-[-]} &=& u_{(1)+} \otimes_\Aopp u_{(1)-} \otimes_\ahha u_{(2)}, \\
\label{mampf3}
u_{[+]} \otimes_\ahha u_{[-]+} \otimes_\Aopp u_{[-]-} &=& u_{(2)[+]} \otimes_\ahha u_{(2)[-]} \otimes_\Aopp u_{(1)}
\end{eqnarray}
hold between left and right Hopf structures.


\providecommand{\bysame}{\leavevmode\hbox to3em{\hrulefill}\thinspace}
\providecommand{\MR}{\relax\ifhmode\unskip\space\fi M`R }
\providecommand{\MRhref}[2]{%
  \href{http://www.ams.org/mathscinet-getitem?mr=#1}{#2}}
\providecommand{\href}[2]{#2}


\begin{thebibliography}{10}

\bibitem[ArKe]{ArmKel:DIOTTTCOAA}
M. Armenta and B. Keller, \emph{Derived invariance of the
  {T}amarkin-{T}sygan calculus of an algebra}, C. R. Math. Acad. Sci. Paris
  \textbf{357} (2019), no.~3, 236--240.

  \bibitem[At]{Ati:CACIFB}
M. Atiyah, \emph{Complex analytic connections in fibre bundles}, Trans. Amer. Math. Soc. \textbf{85} (1957), 181--207.
  
   \bibitem[B\"o]{Boe:HA}
 G. B{\"o}hm, \emph{Hopf algebroids}, Handbook of algebra, {V}ol. 6, North-Holland, Amsterdam,
   2009, pp.~173--236.


\bibitem[B\"oBrzWi]{BoeBrzWis:MACOMC}
G. B{\"o}hm, T. Brzezi{\'n}ski, and R. Wisbauer, \emph{Monads
  and comonads on module categories}, J. Algebra \textbf{322} (2009), no.~5,
  1719--1747.


\bibitem[Brz]{Brz:HCHWCC}
T. Brzezi{\'n}ski, \emph{Hopf-cyclic homology with contramodule
  coefficients}, Quantum groups and noncommutative spaces, Aspects Math., E41,
  Vieweg + Teubner, Wiesbaden, 2011, pp.~1--8.


\bibitem[BrzWi]{BrzWis:CAC}
T.~Brzezi{\'n}ski and R.~Wisbauer, \emph{Corings and comodules}, London
  Mathematical Society Lecture Note Series, vol. 309, Cambridge University
  Press, Cambridge, 2003.

\bibitem[Ca]{Cal:FFLA}
D. Calaque, \emph{Formality for {L}ie algebroids}, Comm. Math. Phys.
  \textbf{257} (2005), no.~3, 563--578.
  
\bibitem[CanWe]{CanWei:GMFNCA}
A. Cannas~da Silva and A. Weinstein, \emph{Geometric models for
  noncommutative algebras}, Berkeley Mathematics Lecture Notes, vol.~10,
  American Mathematical Society, Providence, RI, 1999.



\bibitem[C]{Car:LTDUGDLEDUEFP}
H. Cartan, \emph{La transgression dans un groupe de {L}ie et dans un espace
  fibr\'e principal}, Colloque de topologie (espaces fibr\'es), {B}ruxelles,
  1950, Georges Thone, Li\`ege, 1951, pp.~57--71.


\bibitem[CE]{CarEil:HA}
H. Cartan and S. Eilenberg, \emph{Homological algebra}, Princeton
  University Press, Princeton, N. J., 1956.


\bibitem[Co1]{Con:CCEFE}
A.~Connes, \emph{Cohomologie cyclique et foncteurs {${\rm Ext}\sp n$}}, C. R.
  Acad. Sci. Paris S\'er. I Math. \textbf{296} (1983), no.~23, 953--958.



\bibitem[Co2]{Con:NCDG}
\bysame, 
 \emph{Noncommutative differential geometry}, Inst. Hautes \'Etudes
   Sci. Publ. Math. (1985), no.~62, 257--360.

   \bibitem[Do]{Doi:HC}
Y. Doi, 
\emph{Homological coalgebra}, J. Math. Soc. Japan \textbf{33}
  (1981), no.~1, 31--50.

\bibitem[Do]{Dol:CAEFT}
V. Dolgushev, \emph{Covariant and equivariant formality theorems}, Adv. Math. \textbf{191} (2005), no.~1, 147--177.


\bibitem[DoTaTs1]{DolTamTsy:FOTHCAOHC}
V. Dolgushev, D. Tamarkin, and B. Tsygan, \emph{Formality of the homotopy calculus algebra of Hochschild (co)chains}, preprint (2008), {\tt arXiv:08075117}.

\bibitem[DoTaTs2]{DolTamTsy:FTFHCATA}
  \bysame,
  \emph{{Formality
  theorems for Hochschild complexes and their applications.}}, {Lett. Math.
  Phys.} \textbf{90} (2009), no.~1-3, 103--136.


\bibitem[DoTaTs3]{DolTamTsy:NCATGMC}
 \bysame, 
  \emph{Noncommutative
  calculus and the {G}auss-{M}anin connection}, Higher structures in geometry
  and physics, Progr. Math., vol. 287, Birkh\"{a}user/Springer, New York, 2011,
  pp.~139--158. \MR{2762543}


\bibitem[EiMo]{EilMoo:FORHA}
S. Eilenberg and J. Moore, \emph{Foundations of relative homological
  algebra}, Mem. Amer. Math. Soc. No. \textbf{55} (1965).

\bibitem[FiKo]{FioKow:HBOCANCC}
D. Fiorenza and N. Kowalzig,
\emph{Higher brackets on cyclic and negative cyclic (co)homology}, (2017), to appear in Int. Math. Res. Not. (IMRN), {\tt doi:10.1093/imrn/rny241}.

\bibitem[GeDaTs]{GelDalTsy:OAVONCDG}
I.~Gel'fand, Y.~Daletski{\u\i}, and B.~Tsygan, \emph{On a
  variant of noncommutative differential geometry}, Dokl. Akad. Nauk SSSR
  \textbf{308} (1989), no.~6, 1293--1297.

    \bibitem[Ge]{Ger:TCSOAAR}
  M. Gerstenhaber, \emph{The cohomology structure of an associative ring},
    Ann. of Math. (2) \textbf{78} (1963), no.~2, 267--288.

  
\bibitem[He]{Her:HCOKDP}
E. Herscovich, \emph{Hochschild (co)homology of Koszul dual pairs},
J. Noncommut. Geom. \textbf{13} (2019), no.~1, 59--85. 

\bibitem[HKosR]{HocKosRos:DFORAA}
G.~Hochschild, B.~Kostant, and A.~Rosenberg, \emph{Differential
  forms on regular affine algebras}, Trans. Amer. Math. Soc. \textbf{102}
  (1962), 383--408.



\bibitem[Hue]{Hue:PCAQ}
Johannes Huebschmann, \emph{Poisson cohomology and quantization}, J. Reine
  Angew. Math. \textbf{408} (1990), 57--113.

\bibitem[Ka]{Kas:QG}
C.~Kassel, \emph{Quantum groups}, Graduate Texts in Mathematics, vol. 155,
   Springer-Verlag, New York, 1995.

\bibitem[K]{Kon:DQOPM}
M.~Kontsevich, \emph{Deformation quantization of Poisson manifolds}, 
{Lett. Math. Phys.} \textbf{66} (2003), no.~3, 157--216.


\bibitem[KS1]{KonSoi:DOAOOATDC}
M.~Kontsevich and Y.~Soibelman, \emph{Deformations of algebras over operads
  and the {D}eligne conjecture}, Conf\'erence {M}osh\'e {F}lato 1999, {V}ol.
  {I} ({D}ijon), Math. Phys. Stud., vol.~21, Kluwer Acad. Publ., Dordrecht,
  2000, pp.~255--307.


\bibitem[KS2]{KonSoi:NOAA}
\bysame, \emph{Notes on {$A_\infty$}-algebras,
  {$A_\infty$}-categories and non-commutative geometry}, Homological mirror
  symmetry, Lecture Notes in Phys., vol. 757, Springer, Berlin, 2009,
  pp.~153--219.


   
\bibitem[Ko1]{Kow:GABVSOMOO}
 N. Kowalzig, 
 \emph{Gerstenhaber and Batalin-Vilkovisky structures on modules over operads}, Int.\ Math.\ Res.\ Not.\ \textbf{2015} (2015), no.~22, 11694--11744.

\bibitem[Ko2]{Kow:BVASOCAPB}
\bysame, \emph{Batalin-Vilkovisky algebra structures on 
 ${\rm (Co)}\Tor$ and Poisson bialgebroids}, 
J.\ Pure Appl.\ Algebra \textbf{219} (2015), no.~9, 3781--3822.

\bibitem[Ko3]{Kow:WEIABVA}
\bysame, \emph{When $\Ext$ is a Batalin-Vilkovisky algebra}, J. Noncommut. Geom. \textbf{12} (2018), no.~3, 1081--1131.

 \bibitem[KoKr]{KowKra:BVSOEAT}
N.~Kowalzig and U.~Kr\"ahmer,  
 \emph{{B}atalin-{V}ilkovisky structures on $\Ext$ and $\Tor$}, 
 J. Reine Angew. Math. {\bf 697} (2014), 159--219.



\bibitem[KoPo]{KowPos:TCTOHA}
N.~Kowalzig and H.~Posthuma, \emph{The cyclic theory of {H}opf
  algebroids}, J. Noncomm. Geom. \textbf{5} (2011), no.~3, 423--476.  

\bibitem[La]{Lam:VDBDABVSOCYA}
T. Lambre, \emph{Dualit\'e de van den Bergh 
et structure de Batalin-Vilkoviski\u\i{} sur les alg\`ebres de 
Calabi-Yau}, 
J. Noncommut. Geom. {\textbf 4} (2010), no.~3, 441--457.

\bibitem[Le]{Lei:HOHC}
T.~Leinster, \emph{Higher operads, higher categories}, London
  Mathematical Society Lecture Note Series, vol. 298, Cambridge University
  Press, Cambridge, 2004.

 \bibitem[Lo]{Lod:CH}
 J.-L. Loday, \emph{Cyclic homology}, second ed., Grundlehren Math. Wiss., vol.~301, Springer-Verlag, Berlin, 1998.

  \bibitem[LoVa]{LodVal:AO}
  J.-L. Loday and B. Vallette, \emph{Algebraic operads}, 
  Grundlehren Math. Wiss., vol.~346, Springer-Verlag, Berlin, 2012
 
   \bibitem[McCSm]{McCSmi:ASODHCC}
  J. McClure and J. Smith, \emph{A solution of {D}eligne's
    {H}ochschild cohomology conjecture}, Recent progress in homotopy theory
    ({B}altimore, {MD}, 2000), Contemp. Math., vol. 293, Amer. Math. Soc.,
    Providence, RI, 2002, pp.~153--193.

\bibitem[NeTs]{NesTsy:OTCROAA}
R. Nest and B. Tsygan, \emph{On the cohomology ring of an algebra},
  Advances in geometry, Progr. Math., vol. 172, Birkh\"auser Boston, Boston,
  MA, 1999, pp.~337--370.

\bibitem[Po1]{Pos:HAOSAS}
L. Positselski, \emph{Homological algebra of semimodules and
  semicontramodules}, Instytut Matematyczny Polskiej Akademii Nauk. Monografie
  Matematyczne (New Series), vol.~70,
  Birkh\"auser/Springer Basel AG, Basel, 2010.

  \bibitem[Po2]{Pos:C}
    \bysame, \emph{Contramodules}, preprint (2015), {\tt arXiv:1503.00991}.

    

\bibitem[Pr]{Pra:TDLPLGD}
  J. Pradines, \emph{Th\'eorie de Lie pour les groupo\"ides diff\'erentiables. Calcul diff\'erentiel dans la cat\'egorie des groupo\"ides infinit\'esimaux}, C.\ R.\ Acad.\ Sci.\ Paris S\'er.\ A-B \textbf{264}
  (1967), A245--A248.

\bibitem[Ra]{Rav:CCASHGOS}
D. Ravenel, 
\emph{Complex cobordism and stable homotopy groups of
  spheres}, Pure and Applied Mathematics, vol. 121, Academic Press Inc.,
  Orlando, FL, 1986.
  
\bibitem[Ri]{Rin:DFOGCA}
G. Rinehart, \emph{Differential forms on general commutative algebras},
Trans. Amer. Math. Soc. \textbf{108} (1963), 195--222.
  
\bibitem[Sch]{Schau:DADOQGHA}
P. Schauenburg, \emph{Duals and doubles of quantum groupoids ({$\times\sb
    R$}-{H}opf algebras)}, New trends in Hopf algebra theory (La Falda, 1999),
    Contemp. Math., vol. 267, Amer. Math. Soc., Providence, RI, 2000,
    pp.~273--299.

\bibitem[Tak]{Tak:GOAOAA}
  M. Takeuchi, \emph{Groups of algebras over {$A\otimes \overline A$}}, J.
    Math. Soc. Japan \textbf{29} (1977), no.~3, 459--492.

\bibitem[TaTs]{TamTsy:NCDCHBVAAFC}
 D.~Tamarkin and B.~Tsygan, \emph{Noncommutative differential
    calculus, homotopy {BV} algebras and formality conjectures}, Methods Funct.
    Anal. Topology \textbf{6} (2000), no.~2, 85--100.

  \bibitem[Tam]{Tam:TTTCOAAALS}
    P. Tamaroff,  \emph{The Tamarkin-Tsygan calculus of an algebra \`a la Stasheff},
   preprint (2019), {\tt arXiv:1907.08888}, to appear in Homology, Homotopy and Applications. 

    
\bibitem[Ts1]{Tsy:CH}
B. Tsygan, \emph{Cyclic homology}, Cyclic homology in non-commutative geometry,
  Encyclopaedia Math. Sci., vol. 121, Springer, Berlin, 2004, pp.~73--113.

\bibitem[Ts2]{Tsy:NCAO}
\bysame, \emph{Noncommutative calculus and operads}, Topics in
  noncommutative geometry, Clay Math. Proc., vol.~16, Amer. Math. Soc.,
  Providence, RI, 2012, pp.~19--66.
  
  
\end{thebibliography}
\end{document}